\documentclass[10pt]{amsart}
\usepackage{color}   

\renewcommand\frak{\mathfrak}

\usepackage{amsmath,amsthm,amssymb,url,mathtools,enumerate,colonequals}
\include{xy}
\xyoption{all}
\SelectTips{cm}{}

\DeclareFontFamily{U}{mathx}{}
\DeclareFontShape{U}{mathx}{m}{n}{<-> mathx10}{}
\DeclareSymbolFont{mathx}{U}{mathx}{m}{n}
\DeclareMathAccent{\widehat}{0}{mathx}{"70}
\DeclareMathAccent{\widecheck}{0}{mathx}{"71}

\newcommand{\M}{\mathrm{M}}
\DeclareMathOperator{\Tr}{\mathrm{Tr}}
\DeclareMathOperator{\No}{\mathrm{N}}
\DeclareMathOperator{\Coeff}{coeff}
\newenvironment{enumalph}
{\begin{enumerate}
}
{\end{enumerate}}
\def\myT{t} %generic paramodular index
\theoremstyle{remark}

\newcommand\Z{{\mathbb Z}}
\newcommand\N{{\mathbb N}}
\newcommand\Zpos{\N}

\newcommand\Znn{\N_0}

\newcommand\Zbar{{\ol\Z}}
\newcommand\ol[1]{\overline{#1}}
\newcommand\Q{{\mathbb Q}}

\newcommand\Qbar{{\ol\Q}}

\newcommand\R{{\mathbb R}}

\newcommand\Rpos{\R_{>0}}

\newcommand\C{{\mathbb C}}

\newcommand\F{{\mathbb F}}

\newcommand\lset{\{\,}
\newcommand\rset{\,\}}
\newcommand\lra{\longrightarrow}
\newcommand\bs{\backslash}
\newcommand\inv{^{-1}}
\newcommand\tr{\operatorname{tr}}
\newcommand\ip[2]{\langle #1,#2\rangle} % inner product

\newcommand\siX[1]{{\mathcal X}_{#1}} % positive def semi-integral
\newcommand\Xtwo{\siX2}

\newcommand\XtwoN{\siX2(N)}
\newcommand\XtwoNsemi{\XtwoN^{\rm semi}}

\newcommand\SLgp{\operatorname{SL}}
\newcommand\SLfour{\SLgp(4)}
\newcommand\SL[2]{\SLgp_{#1}(#2)}
\newcommand\SLtwo[1]{\SL2{#1}}
\newcommand\SLtwoZ{\SLtwo\Z}

\newcommand\SLtwoR{\SLtwo\R}

\newcommand\Spgp{\operatorname{Sp}}
\newcommand\Sp[2]{\operatorname{Sp}_{#1}(#2)}
\newcommand\SptwoZ{\Sp2\Z}
\newcommand\SptwoQ{\Sp2\Q}

\newcommand\SptwoR{\Sp2\R}
\newcommand\Gamzero{\Gamma_{\!0}}
\newcommand\GzN{\Gamzero(N)}
\newcommand\GSppos[2]{{\rm GSp}_{#1}^+(#2)}

\newcommand\Ptwoone{\operatorname{P}_{2,1}}

\newcommand\PtwooneQ{\Ptwoone(\Q)}

\newcommand\rmat[4]{\left(\begin{array}{rr}
{#1}&{#2}\\{#3}&{#4}\end{array}\right)}
\newcommand\cmat[4]{\left(\begin{array}{cc}
{#1}&{#2}\\{#3}&{#4}\end{array}\right)}
\newcommand\smallmat[4]{\left(\begin{smallmatrix}
{#1}&{#2}\\{#3}&{#4}\end{smallmatrix}\right)}

\newcommand\smallmatabcd{\smallmat abcd}

\newcommand\paramodulargroup[1]{{\rm K}(#1)}
\newcommand\KN{\paramodulargroup N}

\newcommand\pvar{{\cmat\tau zz\omega}} % paramodular variable

\newcommand\smallpindN{{\smallmat n{r/2}{r/2}{mN}}}

\newcommand\fcJ[2]{c\left(#1;#2\right)} % Fourier-Jacobi coefficient c(t;f)

\newcommand\UHP{{\mathcal H}} % upper half plane
\newcommand\UHPtwo{{\UHP_2}}

\newcommand\wtoperator{|}
\newcommand\wtvar[2]{\wtoperator_{#2}#1}
\newcommand\wtk[1]{\wtvar{#1}k}
\newcommand\wtop[1]{\wtoperator{#1}}

\newcommand\fc[2]{a(#1;#2)} % Fourier coefficient a(t;f)

\newcommand\e{{\rm e}} % exponential function

\newcommand\MFsnoweight{{\mathcal M}} % modular forms
\newcommand\MFswtgp[2]{\MFsnoweight_{#1}(#2)}
\newcommand\MFs[1]{\MFswtgp k{#1}} % modular forms of weight k
\newcommand\MFskN{\MFswtgp k{\paramodulargroup N}}

\newcommand\MFsthreelev[1]{\MFswtgp 3{\paramodulargroup{#1}}}

\newcommand\CFsnoweight{{\mathcal S}} % cusp forms
\newcommand\CFswtgp[2]{\CFsnoweight_{#1}(#2)}
\newcommand\CFs[1]{\CFswtgp k{#1}} % cusp forms of weight k
\newcommand\CFskN{\CFswtgp k{\paramodulargroup N}}

\newcommand\CFsthreeN{\CFswtgp 3{\paramodulargroup N}}
\newcommand\CFsthreelev[1]{\CFswtgp 3{\paramodulargroup{#1}}}

\newcommand\GritSp{{\mathcal G}}

\newcommand\JFind{m} % default Jacobi form index: N or m?
\newcommand\MkKN{\MFswtgp k\KN}

\newcommand\SkKN{\CFswtgp k\KN}

\newcommand\HeckeT{{\mathrm T}}

\newcommand\HT[1]{\HeckeT(#1)}

\newcommand\HTo[1]{\HeckeT_1(#1)}
\newcommand\HTzo[1]{\HeckeT_{0,1}(#1)}
\newcommand\HToz[1]{\HeckeT_{1,0}(#1)}
\newcommand\Tp{\HT p}
\newcommand\Ttwo{\HT2}
\newcommand\Tthree{\HT3}
\newcommand\Tfive{\HT5}
\newcommand\Topsq{\HTo{p^2}}
\newcommand\Tzop{\HTzo p}
\newcommand\Tozpsq{\HToz{p^2}}

\newcommand\evhoef[2]{\lambda_{#2}(#1)} % eigenvalue for Hecke op and eigenform
\newcommand\id[1]{\langle#1\rangle} % ideal notation <...>

\newcommand\Jac[2]{{\rm J}_{#1,#2}}

\newcommand\Jackm{\Jac km}

\newcommand\Jcusp[2]{{\rm J}_{#1,#2}^{\rm cusp}}

\newcommand\JkmNcusp{\Jcusp k{mN}}
\newcommand\JkNcusp{\Jcusp kN}

\newcommand\Jwh[2]{{\rm J}_{#1,#2}^!}

\newcommand\JzeroNwh{\Jwh0N}

\newcommand\Grit{\opratorname Grit}

\def\Grit{\operatorname{Grit}}
\def\Borch{\operatorname{Borch}}

\def\TB{{\rm TB}}

\def\GU{\operatorname{GU}}

\newcommand\mmid{\|} % exactly divides

\renewcommand{\Im}[1]{{\rm Im}(#1)}

\newcommand\BL{\Borch}

\newcommand\idealp{{\mathfrak p}}

\newcommand\ideala{{\mathfrak a}}
\newcommand\idealb{{\mathfrak b}}
\newcommand\idealc{{\mathfrak c}}
\newcommand\idealf{{\mathfrak f}}

\newcommand\idealuu{{\mathfrak u}}
\newcommand\OK{{\mathcal O}_K}

\newcommand\OL{{\mathcal O}_L}

\newcommand{\GSp}{\operatorname{GSp}}

\newcommand{\diag}{\operatorname{diag}}

\newcommand{\floor}{\operatorname{floor}}

\newcommand{\disc}{\operatorname{disc}}

\newcommand\mymat[4]{
{\left(
\begin{smallmatrix}#1&#2\\#3&#4\end{smallmatrix}
\right)}}
\newcommand{\up}{^{{\phantom{2}}^{\phantom{2}}}}
\newcommand{\down}{_{\phantom{\int}}}

\theoremstyle{plain}
\newtheorem{theorem}{Theorem}[section]

\newtheorem{lemma}[theorem]{Lemma}
\newtheorem{proposition}[theorem]{Proposition}
\newtheorem{corollary}[theorem]{Corollary}

\allowdisplaybreaks

\begin{document}

\date{\today}

\title[Weight 3 paramodular eigenvalues and congruences]
{Eigenvalues and congruences for the weight~3 paramodular nonlifts of
  levels~61, 73, and~79}

\author[C.~Poor]{Cris Poor}
\address{Department of Mathematics, Fordham University, Bronx, NY 10458 USA}
\email{poor@fordham.edu}

\author[J.~Shurman]{Jerry Shurman}
\address{Department of Mathematics, Reed College, Portland, OR 97202 USA}
\email{jerry@reed.edu}

\author[D.~Yuen]{David S.~Yuen}
\address{92-1507 Punawainui St., Kapolei, HI 96707 USA}
\email{yuen888@hawaii.edu}

\begin{abstract}
We use Borcherds products to give a new construction of the weight~$3$
paramodular nonlift eigenform~$f_N$ for levels~$N=61,73,79$.
We classify the congruences of~$f_N$ to Gritsenko lifts.
We provide techniques that compute eigenvalues to support
future modularity applications.
Our method does not compute Hecke eigenvalues from Fourier
coefficients but instead uses elliptic modular forms, specifically
the restrictions of Gritsenko lifts and their images under the slash
operator to modular curves.
\end{abstract}


\maketitle

%\tableofcontents

%%%%%%%%%%%%%%%%%%%%%%%%%%%%%%%%%%
%%%%%%%%%%%%%%%%%%%%%%%%%%%%%%%%%%
%%%%%%%%%%%%%%%%%%%%%%%%%%%%%%%%%%
%%%%%%%%%%%%%%%%%%%%%%%%%%%%%%%%%%
\section{Introduction\label{sectionIntr}}
%%%%%%%%%%%%%%%%%%%%%%%%%%%%%%%%%%
%%%%%%%%%%%%%%%%%%%%%%%%%%%%%%%%%%
%%%%%%%%%%%%%%%%%%%%%%%%%%%%%%%%%%
%%%%%%%%%%%%%%%%%%%%%%%%%%%%%%%%%%

We compute Hecke eigenvalues of the nonlift newforms
$f_{N} \in \CFsthreeN$ for the prime levels $N=61,73,79$ to support
future modularity proofs for Calabi--Yau threefolds~with conductor~$N$.   
Vasily Golyshev and Duco von Straten \cite{GvS} have recently
discovered a Calabi--Yau threefold~with conductor~$N=79$.  
The anticipated modularity proof would proceed by showing the 
equivalence of the associated Galois representations, as in~\cite{bpptvy18}.
We also give a new simple construction of each~$f_{N}$ using Borcherds
products.  This construction allows us to prove Fourier coefficient
congruences between lifts and nonlifts.
This article is part of the Day~$61$ project begun in Mainz in January, 2019.    

The first appearance of the aforementioned $f_N$-eigenvalues occurred
in the work of Avner Ash, Paul E.\ Gunnells, and Mark
McConnell~\cite{MR2394820}, who were searching for an element of
the cohomology space $H^5(\Gamma_0(N);\C)$ that genuinely arises from
$\SLfour$.  Here $\Gamma_0(N) \subseteq \SL{4}{\Z}$ is the
subgroup of elements having bottom row congruent to $(0,0,0,*)\bmod N$.
Although they did not find such an element, for $N=61,73,79$ they did see
$2$ and~$3$-Euler factors that they believed originated from
degree~$2$ Siegel modular cusp forms.

Paramodular nonlift newforms in $\CFsthreeN$ were constructed as
holomorphic quotients of Gritsenko lifts in~\cite{py15} for $N=61,73,79$.
Eigenvalues were computed directly from the Fourier coefficients, and enough 
Fourier coefficients were computed to give the $2,3$,~and~$5$-Euler factors. 
The existence of $f_{61}$, for example, follows from the dimension
formula for prime levels due to Tomoyoshi Ibukiyama \cite{i07}, 
namely $\dim\CFsthreelev{61}=7$, and 
the dimension of lifts $\dim \Jcusp{3}{61}=6$ from \cite{ez85}.  
Actually, $61$ is the lowest level, prime or composite, for which 
$\CFsthreeN$ contains nonlifts; the first such levels are
$61, 69, 73, 76, 79, 82, 85, 87, 89$.  

Golyshev and Anton Mellit developed an experimental method, 
relying on the existence of a functional equation, to directly search for 
$L$-series.  In 2010, Mellit found the first $53$ Dirichlet coefficients 
of a degree~$4$ $L$-series with conductor~$61$ and matched the initial
coefficients with the Euler factors for~$L(s,f_{61}, \text{spin})$ in
the arXiv version of~\cite{py15}.  

Supported by increasingly broad dimension formulae, 
Ibukiyama proposed conjectures relating scalar~\cite{i84,i85} 
and vector~\cite{i07b,ik17} paramodular forms for~$\GSp(4)$ 
of weight~$k \ge 3$ to algebraic modular forms on the compact twist  
$\GU(2,B)$, where~$B$ is a definite quaternion algebra.  
A form of the conjecture in~\cite{i07b} has been proven by Pol van Hoften~\cite{MR4329279}.   
The conjecture of Ibukiyama and Hidetaka Kitayama in~\cite{ik17} 
has been proven by Mirko R\"{o}sner and Rainer Weissauer in~\cite{rosnerweissauer2021}, 
and broadened in~\cite{dummiganprt2021} to allow the discriminant of the 
quaternion algebra~$B$ to properly divide the level~$N$.   
The main result of Neil Dummigan, Ariel Pacetti, Gustavo Rama, and Gonzalo Tornaria
in~\cite{dummiganprt2021} is to give a precise correspondence, 
influenced by~\cite{i19}, between algebraic modular forms for~$\GU(2,B)$ 
and orthogonal modular forms for a carefully chosen quinary lattice.  
As a consequence,   the Hecke
eigenvalues of paramodular newforms and certain orthogonal modular forms 
agree for weight~$k \ge 3$ and levels~$N$ 
such that $p \mmid N$ for some prime~$p$.  
Fast methods for computing eigenvalues of
orthogonal modular forms for prime levels~$N$ were developed by a
collaboration of Jeffery Hein, Watson Ladd, and Tornaria
\cite{Hein16,Ladd18}.
Rama and Tornaria~\cite{ramatornaria20} extended these methods 
to orthogonal
modular forms with characters involving the spinor norm, 
and they conjectured a Hecke invariant isomorphism for prime~$N$  
between $S_3\left( K(N) \right)$ and a direct sum of spaces of 
orthogonal modular forms with trivial and nontrivial characters.  
This work was a motivation for~\cite{rosnerweissauer2021}.  
Hein~\cite{Hein16} computed Euler factors for $p<100$ and prime
levels~$N \le 197$.
The appendix of \cite{Ladd18} contains Euler factors for primes
$3\le p \le 31$ and prime levels~$N$ ranging from $61$ to~$359$.
In~\cite{ramatornaria20} the  $2$, $3$, and $5$-Euler factors are given
for squarefree levels~$N<1000$.
%\smallskip
An even larger data set of eigenvalues of orthogonal modular forms for 
levels~$N$ meeting the condition of~\cite{dummiganprt2021} has been 
given by Eran Assaf, Ladd, Rama, Tornaria, and John Voight in~\cite{alrtv2023}.  
This includes the eigenvalues of~$T(p^j)$ for similitudes~$p^j<200$, $1 \le j \le 2$, 
for weight three paramodular newforms.

The only currently known way to associate a Galois representation
to an orthogonal modular form is to pass to a paramodular form
with the same eigenvalues.
Galois representations can be associated to elements of 
$H^5(\Gamma_0(N);\C)$ as in~\cite{MR4078966}, but computing 
eigenvalues using the sharbly complex is expensive.  
Thus the current path to proving the modularity of a Calabi--Yau
threefold of Hodge type~$(1,1,1,1)$ relies on associating Galois
representations to both the threefold and a paramodular newform, and
computing enough Hecke eigenvalues to apply the Faltings--Serre 
method for $\GSp(4)$.  Once modularity is proven, the most efficient
way to compute Euler factors will be by counting points over finite
fields or, perhaps, by the fourth order Picard--Fuchs differential equations associated to the
Calabi--Yau threefolds, see~\cite{cos21} for conjectures on this topic.    
In effect, this article computes eigenvalues
by a slow automorphic method in order to support a future proof that
these same eigenvalues can be computed by a fast arithmetic method.
For $N=61,73,79$ we give $p$-Euler factors for $p$ up to~$19$
and~$\Tp$-eigenvalues for prime~$p<200$, including the bad prime $p=N$.

In this article, Hecke eigenvalues of paramodular newforms are
computed by the technique of restriction to modular curves used
in~\cite{bpptvy18}.
This technique does not use Fourier coefficients, and it also avoids
the multivariable Fourier expansions of paramodular forms that were
used in~\cite{py15}.
The technique instead works with single variable $q$-expansions of
elliptic modular forms that are obtained by restriction of
$f_{N}\wtop\gamma$, for $\gamma\in\GSppos4\Q$, to a fixed
auxiliary modular curve.  The efficiency of this technique is
sensitive to the method by which the paramodular form is constructed.
Here $f_{N}$ is constructed as a quotient of Gritsenko lifts and a
new simple proof of this construction is presented using ideas
from~\cite{gpy15}.
The real virtue of the technique of restriction to modular curves,
however, is that it enables various speed-ups, a few of which were given in~\cite{bpptvy18}.    
This article presents,
for the first time, all known speed-ups to this technique for computing the
eigenvalues of the Hecke operators $\Tp,\Topsq$ for $p\nmid N$ and
$\Tp,\Tozpsq$ for $p\mmid N$, where $p$ is prime.  

Constructing the~$f_N$ as linear combinations of Gritsenko
lifts and Borcherds products allows the direct proof of optimal
congruences among Fourier coefficients.
These congruences of Fourier coefficients immediately imply
congruences of eigenvalues and thus of Euler polynomials.
We are thereby able to classify all the congruences of the
nonlift eigenforms~$f_N$ to Gritsenko lifts.
This method of proving congruences allows the congruence prime
to divide the discriminant of the polynomial of the Hecke field
generator, a condition that holds in the cases of extension degree $7$
and~$5$ in the table immediately below.
For each row of the table, $f_N$ is congruent to a Gritsenko
lift eigenform $g=g(a)\in\CFsthreeN$ where $a\in\Zbar$ is the
$\HeckeT(2)$-eigenvalue of~$g$; with $K=\Q(a)$, the Euler polynomials
of $f_N$ and~$g$ are congruent modulo the ideal~$\ideala$ of the
number ring~$\OK$, and $\ideala$ is the minimal such $\OK$-ideal.
Further, $g$ can be normalized so that the Fourier coefficients
of $f_N$ and~$g$ are congruent modulo~$\ideala$.
The table also shows the discriminant of the polynomial~$q(x)\in\Z[x]$
of~$a$.
In the last row of the table, $v$ is an algebraic integer constructed
from~$a$ as described in the proof of Theorem~\ref{cong79}.
$$
\begin{array}{|c|c|c|c|}
\hline
N & [K:\Q] & \ideala & \operatorname{disc}(q) \\
\hline\hline
61 & 6 & \langle43,a+7\rangle & 2^{14}\cdot3^6\cdot1892022169 \up \\  
\hline
73 & 1 & 3 & 1 \up \\
   & 7 & \langle3,a\rangle\langle13,a+6\rangle
     & 2^{24}\cdot3\cdot13\cdot19^2\cdot37\cdot101\cdot30391 \\
\hline
79 & 2 & \langle2,a+1\rangle & 17 \up \\
   & 5 & \langle2,v\rangle^3 & 2^{16}\cdot4787257 \\
\hline
\end{array}
$$

Using the speed-ups described in this article, we can compute
the following eigenvalues and Euler polynomials.
These are spin polynomials, and their formula is given in
the proof of Lemma~\ref{lemmathree} below.
There is considerable overlap between these results and results
posted at the companion web page~\cite{ramatornaria20webpage}
to~\cite{ramatornaria20}.

For $N=61$ and $f=f_{61}$ the nonlift eigenform,
the pairs $(p,\lambda_f(\Tp))$ for primes~$p<200$ are
\begin{align*}
&(2,-7), (3,-3), (5, 3), (7, -9), (11, -4), (13, -3), (17, 37),(19, -75),  \\
&(23,  10), (29, 212), (31, -6), (37, -88), (41, -3), (43,  547), (47, -147), \\
&(53, -108), (59, -45), (61, 146), (67, -632), (71, -650), (73, 859),
  (79, -978), \\
&(83, 931), (89, -571), (97, 453), \\
&(101,   830), (103, 1246), (107, 707), (109, -378), (113, -225), \\
&(127,  1607), (131, -1399), (137, -861), (139, 1938), (149, 157), \\
&(151,  2356), (157, -414), (163, -11), (167, -1852), (173, -2021),
  (179,  1444), \\
&(181,  442), (191, -366), (193, -2790), (197, -815), (199, -2753).
\end{align*}
The pairs $(p,\lambda_f(\Topsq))$ for primes~$p\le19$ are
$$
(2, 7), (3, -9), (5, -9), (7, -42), (11, 36), (13, -57), (17, 176),
(19, 32),
$$
and the Euler polynomials for~$p\le19$ are as follows.
$$
\begin{array}{|c||c|}
\hline
p\,                    & Q_p(f_{61};x)     \\
\hline\hline
2\up                                
& 1 + 7 x + 24 x^2 + 56 x^3 + 64 x^4   \\
\hline
3\up       
& 1 + 3 x + 3 x^2 + 81 x^3 + 729 x^4    \\  
\hline
5\up       
& 1 - 3 x + 85 x^2 - 375 x^3 + 15625 x^4                \\
\hline
7\up      
& 1 + 9 x + 56 x^2 + 3087 x^3 + 117649 x^4              \\
\hline
11\up      
& 1 + 4 x + 1738 x^2 + 5324 x^3 + 1771561 x^4                  \\
\hline
13\up        
& 1 + 3 x + 1469 x^2 + 6591 x^3 + 4826809 x^4                   \\
\hline
17\up        
& 1 - 37 x + 7922 x^2 - 181781 x^3 + 24137569 x^4                  \\
\hline
19\up        
& 1 + 75 x + 7486 x^2 + 514425 x^3 + 47045881 x^4                 \\
\hline
\end{array}
$$
For $p=61$ the Atkin--Lehner sign and the Euler polynomial are
$\epsilon_{61}=-1$ and
$$
Q_{61}(f_{61};x) = (1 - 61 x) (1 - 84 x + 226981 x^2).
$$
\smallskip

For $N=73$ and $f=f_{73}$,
the pairs $(p,\lambda_f(\Tp))$ for primes~$p<200$ are
\begin{align*}
&(2, -6), (3, -2), (5, 0), (7, 7), (11, -66), (13, 16), (17, 51),
  (19,16), \\
&(23, 153), (29, -18), (31, -185), (37, 220), (41, 60), (43, -458),
(47, 9), \\
&(53, 138), (59, 84), (61, -218), (67, 70), (71, 279), (73, -23), (79, 808), \\
&(83, -492), (89, -240), (97, 1042), \\
&(101,138), (103, -5), (107, -180), (109, 676), (113, -1086), \\
&(127, -104), (131, -2274), (137, 144), (139, -80), (149, 306), \\
&(151, -1721), (157, -320), (163, 1798), (167, 975), (173, 1422),
  (179, -2784), \\
&(181, -1604), (191,  2661), (193, 2017), (197, -96), (199, 2428).
\end{align*}
The pairs $(p,\lambda_f(\Topsq))$ for primes~$p\le19$ are
$$
(2, 6), (3, -9), (5, 0), (7, -110), (11,  162), (13, -284), (17, -90),
(19, -251),
$$
and the Euler polynomials for~$p\le19$ are as follows.
$$
\begin{array}{|c||c|}
\hline
p\,                    & Q_p(f_{73};x)     \\
\hline\hline
2\up                                
& 1 + 6 x + 22 x^2 + 48 x^3 + 64 x^4   \\
\hline
3\up       
& 1 + 2 x + 3 x^2 + 54 x^3 + 729 x^4    \\  
\hline
5\up       
& 1 + 130 x^2 + 15625 x^4                \\
\hline
7\up      
& 1 - 7 x - 420 x^2 - 2401 x^3 + 117649 x^4              \\
\hline
11\up      
& 1 + 66 x + 3124 x^2 + 87846 x^3 + 1771561 x^4                 \\
\hline
13\up        
& 1 - 16 x - 1482 x^2 - 35152 x^3 + 4826809 x^4                   \\
\hline
17\up        
& 1 - 51 x + 3400 x^2 - 250563 x^3 + 24137569 x^4              \\
\hline
19\up        
& 1 - 16 x + 2109 x^2 - 109744 x^3 + 47045881 x^4                 \\
\hline
\end{array}
$$
For $p=73$ the Atkin--Lehner sign and the Euler polynomial are
$\epsilon_{73}=-1$ and
$$
Q_{73}(f_{73};x) = (1 - 73 x) (1 + 97 x + 389017 x^2).
$$
\smallskip

For $N=79$ and $f=f_{79}$,
the pairs $(p,\lambda_f(\Tp))$ for primes~$p<200$ are
\begin{align*}
&(2, -5), (3, -5), (5, 3), (7, 15), (11, 26), (13, -15), (17, -60),
  (19, 32), \\
&(23, 50), (29, 24), (31, 142), (37, -500), (41, 240), (43, -320),
  (47, -105), \\
&(53, 630), (59,25), (61, 194), (67, 420), (71, 111), (73, 460),
  (79, 128), \\
&(83, -90), (89, -1261), (97, 1895), \\
&(101, -1239), (103, -2105), (107, -295), (109, -608), (113, 90), \\
&(127, 2565), (131, 578), (137, -50), (139, -849), (149, -904), \\
&(151, -354), (157, 1980), (163, 840), (167, 2650), (173, -3570),
  (179, 3340), \\
&(181, -708), (191, -695), (193,1000), (197, -40), (199, -141).
\end{align*}
These are consistent with the table of $a_p$ values for the
Calabi--Yau threefold ${\mathbf Y}_{79}$ in~\cite{GvS}, which goes up
to~$p=103$.  The pairs $(p,\lambda_f(\Topsq))$ for primes~$p\le19$ are
$$
%% (2, 2), (3, 4), (5, -10), (7, -24), (11, 0), (13, -21147), (17, -156), (19, -715) %% old 
(2, 2), (3, 4), (5, -10), (7, -24), (11, 0), (13, -158), (17, -156), (19, -712)         %% new
$$
and the Euler polynomials for~$p\le19$ are as follows.
$$
\begin{array}{|c||c|}
\hline
p\,                    & Q_p(f_{79};x)     \\
\hline\hline
2\up                                
& 1 + 5 x + 14 x^2 + 40 x^3 + 64 x^4   \\
\hline
3\up       
& 1 + 5 x + 42 x^2 + 135 x^3 + 729 x^4    \\  
\hline
5\up       
& 1 - 3 x + 80 x^2 - 375 x^3 + 15625 x^4               \\
\hline
7\up      
& 1 - 15 x + 182 x^2 - 5145 x^3 + 117649 x^4              \\
\hline
11\up      
& 1 - 26 x + 1342 x^2 - 34606 x^3 + 1771561 x^4                 \\
\hline
13\up        
%& 1 + 15 x - 272701 x^2 + 32955 x^3 + 4826809 x^4                  \\  %% old
& 1 + 15 x +156 x^2 + 32955 x^3 + 4826809 x^4                  \\           %% new
\hline
17\up        
& 1 + 60 x + 2278 x^2 + 294780 x^3 + 24137569 x^4                  \\
\hline
19\up        
% & 1 - 32 x - 6707 x^2 - 219488 x^3 + 47045881 x^4                  \\  %% old
& 1 - 32 x - 6650 x^2 - 219488 x^3 + 47045881 x^4                  \\      %% new
\hline
\end{array}
$$
For $p=79$ the Atkin--Lehner sign and the Euler polynomial are
$\epsilon_{79}=-1$ and
$$
Q_{79}(f_{79};x) = (1 - 79 x) (1 - 48 x + 493039 x^2).
$$

\noindent{\bf Acknowledgments.\/}  
We are grateful to Neil Dummigan and Dave Roberts for helpful comments.  
This project was completed during a SQuaRE at the American 
Institute for Mathematics. The authors thank AIM for providing a supportive and 
mathematically rich environment.

%%%%%%%%%%%%%%%%%%%%%%%%%%%%%%%%%% 
%%%%%%%%%%%%%%%%%%%%%%%%%%%%%%%%%%
%%%%%%%%%%%%%%%%%%%%%%%%%%%%%%%%%%
%%%%%%%%%%%%%%%%%%%%%%%%%%%%%%%%%%
\section{Background\label{sectionnotation}}
%%%%%%%%%%%%%%%%%%%%%%%%%%%%%%%%%%
%%%%%%%%%%%%%%%%%%%%%%%%%%%%%%%%%%
%%%%%%%%%%%%%%%%%%%%%%%%%%%%%%%%%%
%%%%%%%%%%%%%%%%%%%%%%%%%%%%%%%%%%

The set of positive integers is denoted~$\Zpos$,
the set of nonnegative integers~$\Znn$.
The integer ring of an algebraic number field~$K$ is denoted~$\OK$.
The $\OK$-ideal generated by any set~$S\subseteq\OK$ is denoted~$\id S$,
but we usually write $p\OK$ for~$\id p$ when $p$ is a rational prime.
The ideal norm and the Galois norm from $K$
to~$\Q$ are both denoted~$\No$, so that $\No(\id a)=|\No(a)|$ for all~$a\in\OK$.

\newpage    %%   Use for arXiv.  Remove for journal.  
\subsection{Paramodular forms\label{sectionPara}}
%%%%%%%%%%%%%%%%%%%%%%%%%%%%%%%%%%%%%%%
%%%%%%%%%%%%%%%%%%%%%%%%%%%%%%%%%%%%%%%
%%%%%%%%%%%%%%%%%%%%%%%%%%%%%%%%%%%%%%%
%%%%%%%%%%%%%%%%%%%%%%%%%%%%%%%%%%%%%%%

%%%%%%%%%%%%%%%%%%%%%%%%%%%%%%%%%%%%%%%
%%%%%%%%%%%%%%%%%%%%%%%%%%%%%%%%%%%%%%%
\subsubsection{Definitions, Fourier series representation\label{sectionParaDefs}}
%%%%%%%%%%%%%%%%%%%%%%%%%%%%%%%%%%%%%%%
%%%%%%%%%%%%%%%%%%%%%%%%%%%%%%%%%%%%%%%

%We introduce notation and terminology for paramodular forms.
The degree~$2$ symplectic group $\Spgp(2)$ of $4\times4$ matrices is
defined by the condition $g'Jg=J$, where the prime denotes matrix
transpose and $J$ is the skew form $\smallmat0{-1}1{\phantom{-}0}$
with each block $2\times2$.
The Klingen parabolic subgroup of~$\Spgp(2)$ is
$$
\Ptwoone=\lset
\left(\begin{array}{cc|cc}
* & 0 & * & * \\ * & * & * & * \\
\hline
* & 0 & * & * \\ 0 & 0 & 0 & *
\end{array}\right)\rset,
$$
with either line of three zeros forcing the remaining two
because the matrices are symplectic.
For any positive integer~$N$, the paramodular group $\KN$ of degree~$2$
and level~$N$ is the group of rational symplectic matrices that
stabilize the column vector lattice $\Z\oplus\Z\oplus\Z\oplus N\Z$.  In
coordinates,
$$
\KN=\lset
\left(\begin{array}{cc|cc}
* & *N & * & * \\ * & * & * & */N \\
\hline
* & *N & * & * \\ *N & *N & *N & *
\end{array}\right)\in\SptwoQ:\text{all $*$ entries integral}
\rset.
$$

Let $\UHPtwo$ denote the Siegel upper half space of $2\times2$
symmetric complex matrices that have positive definite imaginary part,
generalizing the complex upper half plane~$\UHP$.
Elements of this space are written
$$
\Omega=\pvar\in\UHPtwo,
$$
with $\tau,\omega\in\UHP$, $z\in\C$, and $\Im\Omega>0$.
Also, letting $\e(w)=e^{2\pi iw}$ for~$w\in\C$, our standard notation
throughout is
$$
q=\e(\tau),\quad\zeta=\e(z),\quad\xi=\e(\omega).
$$
The real symplectic group $\SptwoR$ acts on~$\UHPtwo$ as fractional
linear transformations, $g(\Omega)=(a\Omega+b)(c\Omega+d)\inv$ for
$g=\smallmatabcd$, and the automorphy factor is
$j(g,\Omega)=\det(c\Omega+d)$.  Fix an integer~$k$.
Any function $f:\UHPtwo\lra\C$ and any real symplec\-tic matrix
$g\in\SptwoR$ combine to form another such function through the
weight~$k$ slash operator,
$$
(f\wtk g)(\Omega)=j(g,\Omega)^{-k}f(g(\Omega)).
$$
When $k$ is well established, we freely write $f\wtop g$ rather than
$f\wtk g$.
A paramodular form of weight~$k$ and level~$N$ is a holomorphic
function $f:\UHPtwo\lra\C$ that is $\wtk{\KN}$-invariant.
The space of weight~$k$, level~$N$ paramodular
forms is denoted~$\MkKN$.

A paramodular form of level $N$ has a Fourier expansion
$$
f(\Omega)=\sum_{t\in\XtwoNsemi}\fc tf\,\e(\ip t\Omega)
$$
with all $\fc tf\in\C$, where the index set is
$$
\XtwoNsemi=\lset\smallpindN:n,m\in\Znn,r\in\Z,4nmN-r^2\ge0\rset
$$
and $\ip t\Omega=\tr(t\Omega)$.
For any subring~$R$ of~$\C$ we let $\MkKN(R)$ denote the $R$-module of
$\MkKN$-elements whose Fourier coefficients all lie in~$R$.
This notation also applies to all subspaces of~$\MkKN$ to be
introduced below, e.g., $\SkKN^+(\Z)$ is the $\Z$-module of
paramodular cusp forms that are Fricke eigenfunctions having
eigenvalue~$1$ and all Fourier coefficients in~$\Z$.
An element of $\MkKN(R)$ is said to have {\em unit content\/} if the
ideal generated by its Fourier coefficients is all of~$R$.

The Siegel $\Phi$ operator takes any holomorphic function that has a
Fourier series of the form
$f(\Omega)=\sum_t\fc tf\,\e(\ip t\Omega)$, summing over rational
positive semidefinite $2\times2$ matrices~$t$, to the function
$(\Phi f)(\tau)=\lim_{\lambda\to+\infty}f(\smallmat\tau00{i\lambda})$.
A paramodular form $f$ in~$\MkKN$ is called a cusp form if
$\Phi(f\wtk g)=0$ for all~$g\in\SptwoQ$.
This is a finite condition because 
it only needs to be checked for one representative~$g$ of each double
coset in Helmut Reefschl\"ager's decomposition (\cite{reefschlager73},
and see Theorem~1.2 of~\cite{py13}),
in which a superscript asterisk denotes matrix inverse-transpose,
$$
\SptwoQ=\bigsqcup_{m\in\Zpos:m\mid N}\KN u(\alpha_m)\PtwooneQ,
\qquad\alpha_m=\smallmat1m01,
\ u(\alpha)=\smallmat\alpha00{\alpha^*}.
$$
A paramodular form is a cusp form if and only if its Fourier expansion
is supported on $\XtwoN$, defined by the strict inequality
$4nmN-r^2>0$; this characterization of cusp forms does not hold in
general for groups commensurable with~$\SptwoZ$, but it does hold
for~$\KN$ because the representatives in Reefschl\"ager's decomposition 
have block diagonal form.
The space of paramodular cusp forms is denoted~$\SkKN$.

%%%%%%%%%%%%%%%%%%%%%%%%%%%%%%%%%%%%%%%
%%%%%%%%%%%%%%%%%%%%%%%%%%%%%%%%%%%%%%%
\subsubsection{Symmetric and antisymmetric forms\label{sectionSymm}}
%%%%%%%%%%%%%%%%%%%%%%%%%%%%%%%%%%%%%%%
%%%%%%%%%%%%%%%%%%%%%%%%%%%%%%%%%%%%%%%

%(This paragraph introduces the Fricke involution but not
%Atkin--Lehner involutions.)
The elliptic Fricke involution
$$
\alpha_N=\frac1{\sqrt N}\rmat0{-1}N{\phantom{-}0}:
\tau\longmapsto-\frac1{N\tau}
$$
normalizes the level~$N$ Hecke subgroup $\GzN$ of~$\SLtwoZ$, and it
squares to~$-1$ as a matrix, hence to the identity as a transformation.
The corresponding paramodular Fricke involution is
$$
\mu_N=\cmat{\alpha_N^*}00{\alpha_N}:
\pvar\longmapsto\cmat{\omega N}{-z}{-z}{\tau/N}.
$$
The paramodular Fricke involution normalizes the paramodular group $\KN$,
and it squares to the identity as a transformation.
The space $\CFs\KN$ decomposes as the direct sum of the Fricke eigenspaces
for the two eigenvalues~$\pm1$, $\SkKN=\SkKN^+\oplus\SkKN^-$.
We let $\epsilon$ denote either eigenvalue.
A paramodular Fricke eigenform is called {\em symmetric\/} if
$(-1)^k\epsilon=+1$, and {\em antisymmetric\/} if $(-1)^k\epsilon=-1$.

%%%%%%%%%%%%%%%%%%%%%%%%%%%%%%%%%%%%%%%
%%%%%%%%%%%%%%%%%%%%%%%%%%%%%%%%%%%%%%%
\subsubsection{Atkin--Lehner involutions\label{sectionAL}}
%%%%%%%%%%%%%%%%%%%%%%%%%%%%%%%%%%%%%%%
%%%%%%%%%%%%%%%%%%%%%%%%%%%%%%%%%%%%%%%

%We review Atkin--Lehner involutions, which include the Fricke involution.
Let~$N$ be a positive integer, and let~$c$ be a positive divisor
of~$N$ such that $\gcd(c,N/c)=1$.
In this article $N$ is always squarefree, so $c$ can be any positive
divisor of~$N$.
For any integers $\alpha,\beta,\gamma,\delta$ such that
$\alpha\delta c-\beta\gamma N/c=1$,
an elliptic $c$-Atkin--Lehner matrix is
$$
\alpha_c=\frac1{\sqrt c}\cmat{\alpha c}\beta{\gamma N}{\delta c}.
$$
Especially, for $c=1$ we may take $\alpha,\delta=1$ and
$\beta,\gamma=0$ to get the identity matrix, and for $c=N$ we may take
$\alpha,\delta=0$ and $\beta,\gamma=\mp1$ to get the Fricke involution
matrix $\alpha_N=\frac1{\sqrt N}\smallmat0{-1}N{\phantom{-}0}$.
%Let $\Gamma_0(N)$ denote the level~$N$ Hecke subgroup of~$\SLtwoZ$.
By quick calculations, the inverse of any~$\alpha_c$ is
another~$\tilde\alpha_c$, any product $\tilde\alpha_c\alpha_c$ lies in
$\Gamma_0(N)$ and so the set of all~$\tilde\alpha_c$ lies in the
coset $\Gamma_0(N)\alpha_c$, and this coset also lies in the set of
all~$\tilde\alpha_c$, making them equal.  Consequently, $\alpha_c$
squares into~$\Gamma_0(N)$ and normalizes~$\Gamma_0(N)$.  
%  Now, with an asterisk denoting the matrix inverse-transpose,
%%  asterisk already defined above 
A paramodular $c$-Atkin--Lehner matrix is
$$
\mu_c=u\left(\alpha_c^*\right)= \cmat{\alpha_c^*}00{\alpha_c}.
$$
The inverse of any~$\mu_c$ is another $\tilde\mu_c$, and any product
$\tilde\mu_c\mu_c$ lies in~$\KN$, so that the set of all~$\tilde\mu_c$
lies in the coset $\KN\mu_c$ and they all give the same action on
paramodular forms, although now the containment is proper.
Again $\mu_c$ squares into~$\KN$, and a blockwise check shows that
$\mu_c$ normalizes~$\KN$.
For $c=1$ we take $\mu_1=1_4$.
For $c=N$, the paramodular Fricke involution is
$\mu_N$ from the previous paragraph.

\subsection{Hecke operators\label{sectionHeckeOps}}

The real symplectic group lies in the larger group $\GSppos2\R$
defined by the condition $g'Jg=m(g)J$ with $m(g)\in\Rpos$.
The weight~$k$ slash operator extends to
$(f\wtk g)(\Omega)=m(g)^{2k-3}j(g,\Omega)^{-k}f(g(\Omega))$
for $g\in\GSppos2\R$.
This is the arithmetic normalization of the slash operator, as
compared to the analytic normalization that has $m(g)^k$ instead.
Three Hecke operators figure in this article, defined in the usual way
by double cosets.  In particular, let $\HT{a,b,c,d}$ abbreviate the
double coset $\KN\diag(a,b,c,d)\KN$.
\begin{itemize}
\item $\Tp=\HT{1,1,p,p}$.  Because the double coset here is also
  $\HT{p,p,1,1}$, used to define the operator $\Tzop$
  from~\cite{schmidt18}, the two operators are equal and we make no
  reference to the notation $\Tzop$ in this article.
  In general, $\HT{a,b,c,d}$, $\HT{c,b,a,d}$, $\HT{a,d,c,b}$,
  $\HT{c,d,a,b}$ are all equal, i.e., we may exchange the first and
  third entries, or the second and fourth, or both pairs.
\item $\Topsq=\HT{1,p,p^2,p}$.  This is a standard operator, used
  in~\cite{bpptvy18}.
\item $\Tozpsq=\HT{p,p^2,p,1}$.  This operator is denoted $\HToz p$
  in~\cite{schmidt18}.  We write it $\Tozpsq$ to indicate its
  multiplier.
  The operators $\Topsq$ and $\Tozpsq$ are conjugate under the Fricke
  operator, $\mu_N\inv\Topsq\mu_N=\Tozpsq$.  In general,
  $\mu_N\inv\HT{a,b,c,d}\mu_N=\HT{b,a,d,c}$, and then the right side
  has three other names as explained just above.  
\end{itemize}
In this article we use $\Topsq$ when $p \nmid N$ and $\Tozpsq$ when $p \mmid N$.  

%%%%%%%%%%%%%%%%%%%%%%%%%%%%%%%%%%%%%%% lulu
%%%%%%%%%%%%%%%%%%%%%%%%%%%%%%%%%%%%%%%
\subsubsection{Fourier--Jacobi expansion\label{sectionFour}}
%%%%%%%%%%%%%%%%%%%%%%%%%%%%%%%%%%%%%%%
%%%%%%%%%%%%%%%%%%%%%%%%%%%%%%%%%%%%%%%

The Fourier--Jacobi expansion of a para\-mod\-u\-lar cusp form
$f\in\SkKN$ is
$$
f(\Omega)=\sum_{m\ge1}\phi_m(f)(\tau,z)\xi^{mN},
\quad \Omega=\pvar,
$$
with Fourier--Jacobi coefficients
$$
\phi_m(f)(\tau,z)=\sum_{t=\smallpindN\in\XtwoN}\fc tfq^n\zeta^r.
$$
The Fourier--Jacobi coefficients are also written
$$
\phi_m(f)(\tau,z)=\sum_{n,r:4nmN-r^2>0}\fcJ{n,r}{\phi_m}q^n\zeta^r.
$$
Each Fourier--Jacobi coefficient $\phi_m(f)$ lies in the space
$\JkmNcusp$ of weight~$k$, index~$mN$ Jacobi cusp forms, whose
dimension is known.  Jacobi forms will briefly be reviewed next.

%%%%%%%%%%%%%%%%%%%%%%%%%%%%%%%%%%%%%%%
%%%%%%%%%%%%%%%%%%%%%%%%%%%%%%%%%%%%%%%
%%%%%%%%%%%%%%%%%%%%%%%%%%%%%%%%%%%%%%%
%%%%%%%%%%%%%%%%%%%%%%%%%%%%%%%%%%%%%%%
\subsection{Jacobi forms\label{sectionJaco}}
%%%%%%%%%%%%%%%%%%%%%%%%%%%%%%%%%%%%%%%
%%%%%%%%%%%%%%%%%%%%%%%%%%%%%%%%%%%%%%%
%%%%%%%%%%%%%%%%%%%%%%%%%%%%%%%%%%%%%%%
%%%%%%%%%%%%%%%%%%%%%%%%%%%%%%%%%%%%%%%

\newcommand\JFtoSMF[2]{E_{#2}{#1}}
\newcommand\JFtoSMFm[1]{\JFtoSMF{#1}m}

For the theory of Jacobi forms, see \cite{ez85,gn98,sz89}.
Let $k$ be an integer and let $\JFind$ be a nonnegative integer.  The
complex vector spaces of weight~$k$, index~$\JFind$ Jacobi forms,
Jacobi cusp forms, and weakly holomorphic Jacobi forms consist of
holomorphic functions $g:\UHP\times\C\lra\C$ that have Fourier series
representations
$$
g(\tau,z)=\sum_{n,r}\fcJ{n,r}gq^n\zeta^r
$$
with all $\fcJ{n,r}g\in\C$,
and that satisfy transformation laws and constraints on the support.
With the usual notation $\gamma(\tau)=(a\tau+b)/(c\tau+d)$
and $j(\gamma,\tau)=c\tau+d$ for $\gamma=\smallmatabcd\in\SLtwoZ$ and
$\tau\in\UHP$, the transformation laws are
\begin{itemize}
\item $g(\gamma(\tau),z/j(\gamma,\tau))
=j(\gamma,z)^k\e(\JFind cz^2/j(\gamma,\tau))g(\tau,z)$ for all
$\gamma\in\SLtwoZ$,
\item $g(\tau,z+\lambda \tau+\mu)
=\e(-\JFind \lambda^2\tau-2\JFind \lambda z)g(\tau,z)$ for all
$\lambda,\mu\in\Z$.
\end{itemize}
To describe the constraints on the support,
associate to any integer pair $(n,r)$ the discriminant
$$
D=D(n,r)=4n\JFind-r^2.
$$
The principal part of~$g$ is $\sum_{n<0}g_n(\zeta)q^n$ where
$g_n(\zeta)=\sum_r\fcJ{n,r}g\zeta^r$, and the singular part is
$\sum_{D(n,r)\le0}\fcJ{n,r}gq^n\zeta^r$.
\begin{itemize}
\item For the space $\Jac k\JFind$ of Jacobi forms, if $\JFind>0$ then
  the sum is taken over integers $n$ and~$r$ such that $D\ge0$, so
  that in particular $n\ge0$, and if $\JFind=0$ then the sum is taken
  over $n\in\Znn$ and $r=0$, and we have elliptic modular forms.
\item For the space $\Jcusp k\JFind$ of Jacobi cusp forms, if $\JFind>0$ then
  the sum is taken over integers $n$ and~$r$ such that $D>0$, so
  that in particular $n>0$, and if $\JFind=0$ then the sum is taken
  over $n\in\Zpos$ and $r=0$, and we have elliptic cusp forms.
\item For the space $\Jwh k\JFind$ of weakly holomorphic Jacobi forms
  the sum is taken over integers $n\gg-\infty$ and~$r$.
  For positive index~$\JFind$, the conditions $n\gg-\infty$ and
  $D\gg-\infty$ are equivalent.
\end{itemize}
When all the Fourier coefficients lie in a ring we also append the
ring to the notation; for example, $\Jwh 0\JFind(\Z)$ denotes the
$\Z$-module of weight~$0$, index~$\JFind$ weakly holomorphic forms
with integral Fourier coefficients.

%%%%%%%%%%%%%%%%%%%%%%%%%%%%%%%%%%%%%%%
%%%%%%%%%%%%%%%%%%%%%%%%%%%%%%%%%%%%%%%
%%%%%%%%%%%%%%%%%%%%%%%%%%%%%%%%%%%%%%%
%%%%%%%%%%%%%%%%%%%%%%%%%%%%%%%%%%%%%%%
\subsection{Theta blocks\label{sectionThet}}
%%%%%%%%%%%%%%%%%%%%%%%%%%%%%%%%%%%%%%%
%%%%%%%%%%%%%%%%%%%%%%%%%%%%%%%%%%%%%%%
%%%%%%%%%%%%%%%%%%%%%%%%%%%%%%%%%%%%%%%
%%%%%%%%%%%%%%%%%%%%%%%%%%%%%%%%%%%%%%%

The theory of theta blocks is due to Gritsenko, Skoruppa, and Zagier
\cite{gsz}.
Recall the Dedekind eta function $\eta:\UHP\lra\C$
%$\eta\in\Jcusp{1/2}0(\epsilon)$
and the odd Jacobi theta function
$\vartheta:\UHP\times\C\lra\C$,
%$\vartheta\in\Jcusp{1/2}{1/2}(\epsilon^3 v_H)$,
\begin{align*}
\eta(\tau)&=q^{1/24}\prod_{n\ge1}(1-q^n),\\
\vartheta(\tau,z)
&=\sum_{n\in\Z}(-1)^nq^{(n+1/2)^2/2}\zeta^{n+1/2}\\
&=q^{1/8}(\zeta^{1/2}-\zeta^{-1/2})
\prod_{n\ge1}(1-q^n\zeta)(1-q^n\zeta\inv)(1-q^n).
\end{align*}
For any~$d\in\Zpos$ define $\vartheta_d$
%$\vartheta_d\in\Jcusp{1/2}{d^2/2}(\epsilon^3 v_H^d)$
to be $\vartheta_d(\tau,z)=\vartheta(\tau,dz)$.
Given $k,\ell,d_1,\dotsc,d_\ell\in\Zpos$, the resulting theta block is
defined to be
$$
\TB_k[d_1,\dotsc,d_{\ell}]=\eta^{2k-\ell}\prod_{j=1}^{\ell}\vartheta_{d_j}.
$$
We will use the following result from~\cite{gsz}.

\begin{theorem}[Gritsenko, Skoruppa, Zagier]\label{GSZthetablockthm}
Let $k,m,\ell,d_1,\dotsc,d_\ell\in\Zpos$.
Let $\bar B_2(x)=B_2(x-\lfloor x\rfloor)$ where $B_2(x)=x^2-x+1/6$ is
the second Bernoulli polynomial.
Then $\TB_k[d_1,\dotsc,d_\ell]\in\Jcusp{k}{m}$ if and only if 
$12\mid k+\ell$, $2m = \sum_{j=1}^{\ell}d_j^2$, and 
$\frac{k}{12}+\frac{1}{2}\sum_{j=1}^{\ell} {\bar B}_2(d_j x)>0$
for $0 \le x \le 1$; the positivity needs to be checked only for
$x\in[0,1/2]\cap\frac1{2m}\Z$.
\end{theorem}

\subsection{Gritsenko lifts\label{sectionGlift}}

The Gritsenko lift, or additive lift, \cite{MR1336601} is an injection
$$
\Grit: \JkNcusp\lra\CFskN^{\epsilon}
\quad\text{where $\epsilon = (-1)^k$}.
$$
Its definition uses the Eichler--Zagier \cite{ez85} index raising operator
$V_{\ell}:\Jackm\lra\Jac{k}{m\ell}$,
extended to weakly holomorphic Jacobi forms \cite{gn98}.
The Gritsenko lift of~$\phi\in\JkNcusp$ is
$$
\Grit(\phi)\smallmat{\tau}{z}{z}{\omega} = \sum_{m=1}^{\infty} \left(
\phi \wtop{V_m} \right)(\tau,z)\e(m\omega).
$$

\subsection{Borcherds products\label{sectionBorc}}

The Borcherds product theorem, quoted here from \cite{MR3713095},
is a special case of Theorem~3.3 of \cite{gpy15}, which
in turn is quoted from \cite{gn98,grit12} and relies on the work
of Richard Borcherds.
In the theorem, $\sigma_0(m)$ denotes the number of positive divisors
of the positive integer~$m$.

\begin{theorem}\label{BPthm}
Let $N$ be a positive integer.  Let $\psi\in\JzeroNwh$ be a weakly
holomorphic weight~$0$, index~$N$ Jacobi form, having Fourier expansion
$$
\psi(\tau,z)=\sum_{\substack{n,r\in\Z\\n\gg-\infty}}c(n,r)q^n\zeta^r
\quad\text{where }q=\e(\tau),\ \zeta=\e(\zeta).
$$
Define
\begin{alignat*}2
A&=\frac1{24}c(0,0)+\frac1{12}\sum_{r\ge1}c(0,r),
 &\quad 
B\phantom{_0}&=\frac12\sum_{r\ge1}r\,c(0,r),\\
C&=\frac12\sum_{r\ge1}r^2c(0,r),
 &\quad
D_0&=\sum_{n\le-1}\sigma_0(|n|)c(n,0).
\end{alignat*}
Suppose that the following conditions hold:
\begin{enumerate}
\item $c(n,r)\in\Z$ for all integer pairs $(n,r)$ such that $4nN-r^2\le0$,
\item $A\in\Z$,
\item $\sum_{i\ge1}c(i^2nm,ir)\ge0$ for all primitive integer
  triples $(n,m,r)$ such that $4nmN-r^2<0$ and $m\ge0$.
\end{enumerate}
Then for weight $k=\frac12c(0,0)$ and Fricke eigenvalue
$\epsilon=(-1)^{k+D_0}$, the Borcherds product $\BL(\psi)$ lies
in~$\MFs\KN^\epsilon$.
For sufficiently large $\lambda$, for $\Omega=\smallmat\tau
zz\omega\in\UHPtwo$ and $\xi=\e(\omega)$, the Borcherds product has
the following convergent product expression on the subset
$\lset\Im \Omega>\lambda I_2\rset$ of~$\UHPtwo$:  
$$
\BL(\psi)(\Omega)=
q^A \zeta^B \xi^C 
\prod_{\substack{n,m,r\in \Z,\ m\ge0\\
\text{if $m=0$ then $n\ge0$}\\
\text{if $m=n=0$ then $r<0$}}}
(1-q^n\zeta^r\xi^{mN})^{c(nm,r)}.
$$
Also, let $\lambda(r)=c(0,r)$ for~$r\in\Znn$, and recall the
corresponding theta block,
$$
\TB(\lambda)(\tau,z)=\eta(\tau)^{\lambda(0)}
\prod_{r\ge1}(\vartheta_r(\tau,z)/\eta(\tau))^{\lambda(r)}
\qquad\text{where }\vartheta_r(\tau,z)=\vartheta(\tau,rz).
$$
On $\lset\Im \Omega>\lambda I_2\rset$ the Borcherds product is a
rearrangement of a convergent infinite series,
$$
\BL(\psi)(\Omega)
=\TB(\lambda)(\tau,z)\xi^C\exp\left(-\Grit(\psi)(\Omega)\right).  
$$
\end{theorem}

In the theorem, the divisor of the Borcherds product $\BL(\psi)$ is a sum
of Humbert surfaces with multiplicities, the multiplicities
necessarily nonnegative for holomorphy.  Let $\KN^+$ denote the
group generated by $\KN$ and the paramodular Fricke involution~$\mu_N$.
The sum in item~(3) of the theorem is the multiplicity of
the following Humbert surface in the divisor,
in which $t_o=\smallmat{n}{r/2}{r/2}{mN}$ and $D=4nmN-r^2<0$,
$$
H_N(t_o)=H_N(-D,r)=\KN^+\lset\Omega\in\UHPtwo:\ip\Omega{t_o}=0\rset
\subseteq\KN^+\bs\UHPtwo.
$$
This surface depends only on the discriminant~$D$ and on~$r$,
so that we may take $t_o$ with $m=1$, and furthermore, by work of
Gritsenko and Hulek \cite{gh98}, it depends only on the residue
class of~$r$ modulo~$2N$.
%

%%%%%%%%%%%%%%%%%%%%%%%%%%%%%%%%%%
%%%%%%%%%%%%%%%%%%%%%%%%%%%%%%%%%%
%%%%%%%%%%%%%%%%%%%%%%%%%%%%%%%%%%
%%%%%%%%%%%%%%%%%%%%%%%%%%%%%%%%%%
\section{Construction of the nonlift newforms $f_{N}$\label{sectionconstruction}}
%%%%%%%%%%%%%%%%%%%%%%%%%%%%%%%%%%
%%%%%%%%%%%%%%%%%%%%%%%%%%%%%%%%%%
%%%%%%%%%%%%%%%%%%%%%%%%%%%%%%%%%%
%%%%%%%%%%%%%%%%%%%%%%%%%%%%%%%%%%

The following proposition,
other than its last statement,
was proven in~\cite{py15} using the method of integral closure.
Here we give a new proof using Borcherds products,
as indicated in the added last statement.
This new proof works directly in the weight~$3$ space
with no need to span the weight~$6$ space as the integral closure
method did.
As a byproduct of this construction, $f_{61}$ is clearly congruent to
a Gritsenko lift modulo~$43$.

\begin{proposition}
\label{prop61}
There is a nonlift Hecke eigenform $f_{61}\in\CFsthreelev{61}^{-}(\Z)$ 
with unit content, given by a rational function of Gritsenko lifts,
$$
f_{61} = -9G[1] -2G[2] +22G[3] +9G[4] -10G[5] +19G[6] 
-43\dfrac{G[1]G[6]}{G[2]}.
$$
Here $G[j]=\Grit\left(\TB_3(D_j)\right)$ for $j=1,\dotsc,6$ are
Gritsenko lifts of theta blocks, with
\begin{alignat*}2
D_1&=[2, 2, 2, 3, 3, 3, 3, 5, 7]\qquad
&D_2&=[2, 2, 2, 2, 3, 4, 4, 4, 7]\\
D_3&=[2, 2, 2, 2, 3, 3, 4, 6, 6]\qquad
&D_4&=[1, 2, 3, 3, 3, 3, 4, 4, 7]\\
D_5&=[1, 2, 3, 3, 3, 3, 3, 6, 6]\qquad
&D_6&=[1, 2, 2, 2, 4, 4, 4, 5, 6].
\end{alignat*}
The set $\{f_{61},G[1],\dotsc,G[6]\}$ is a basis of~$\CFsthreelev{61}$.
Each Gritsenko lift $G[j]$ is also a Borcherds product
$B[j]=\BL(-(\TB_3(D_j) \wtop{V_2})/\TB_3(D_j))$. %for~$j=1,\dotsc,6$.
\end{proposition}

The existence of $f_{61}$ follows from the dimension formula of
Ibukiyama for prime levels~$p$ in~\cite{i07}, stated below for convenience.
A more general dimension formula for weights $k\ge3$ and squarefree
levels $N\ge3$ is due to Ibukiyama and Kitayama~\cite{ik17}.

\begin{theorem}[Ibukiyama dimension formula]
Let $p\ge5$ be prime.  Then
\begin{align*}
\dim\CFsthreelev p
&=\frac{1}{2880}(p^2-1)-1 \\
&\quad+\frac{1}{64}(p+1)\left(1-\left(\frac{-1}{p}\right)\right)
      +\frac{5}{192}(p-1)\left(1+\left(\frac{-1}{p}\right)\right) \\
&\quad+\frac{1}{72}(p+1)\left(1-\left(\frac{-3}{p}\right)\right)
      +\frac{1}{36}(p-1)\left(1+\left(\frac{-3}{p}\right)\right) \\
&\quad+\frac{1}{8}\left(1-\left(\frac{2}{p}\right)\right) %\\ 
      +\frac{1}{5}\left(1-\left(\frac{5}{p}\right)\right) %\\ 
+\left\{\frac{1}{6}\ \text{\rm if $ p \equiv 5 \bmod 12$}\right\}.
\end{align*}
\end{theorem}
Ibukiyama's formula gives $\dim\CFsthreelev{61}=7$, and
the dimension formula for Jacobi forms in~\cite{ez85},
$$
\dim \Jcusp{3}{m}
=\sum_{j=1}^{m-1} 
\left( \dim\CFswtgp{2+2j}{\SLtwoZ}-\floor\left( {j^2}/{(4m)} \right) \right),
$$
gives $\dim\Jcusp{3}{61}=6$.
We begin spanning $\CFsthreelev{61}$
by using theta blocks to span $\Jcusp{3}{61}$.
With the $D_j$ from Proposition~\ref{prop61},
Theorem~\ref{GSZthetablockthm} shows that $\TB_3(D_j)\in\Jcusp{3}{61}$
for $j=1,\dotsc,6$.
By computing Fourier coefficients \cite{yuen22},
we see that the theta blocks $\TB_3(D_j)$ form a basis of~$\Jcusp{3}{61}$
and the Gritsenko lifts $G[j]=\Grit\left(\TB_3(D_j)\right)$ a basis of
the Gritsenko lift subspace in $\CFsthreelev{61}$.
We finish spanning~$\CFsthreelev{61}$ by using a special case of the
Borcherds Products Everywhere theorem.  The following result is
Theorem~6.6 of~\cite{gpy15} specialized to the case where the
$q$-order of vanishing is~$1$, and so we call it a corollary.

\begin{corollary}
\label{bpecor}  
Let $k,N,\ell,d_1,\dotsc,d_\ell\in\Zpos$.  Assume  
$\phi= \TB_k[d_1, \dotsc, d_{\ell}] \in \Jcusp{k}{N}$ and $k+\ell=12$.    
Let $\psi= -\frac{\phi \wtop{V_2}}{\phi}$.
Then $\psi\in\Jwh{0}{N}(\Z)$, 
$\BL(\psi) \in \MFskN^{\epsilon}$ for $\epsilon = (-1)^k$, and 
$\BL(\psi)$ and $\Grit(\phi)$ have equal first and second
Fourier--Jacobi coefficients.
\end{corollary}

For $j=1,\dotsc,6$ set $\psi_j=-(\TB_3(D_j) \wtop{V_2})/\TB_3(D_j)$
and $B[j]=\BL(\psi_j)$.  Thus $\psi_j \in \Jwh{0}{61}(\Z)$
and $B[j]\in\MFsthreelev{61}$ by the corollary.
The divisor of~$B[j]$ consists of Humbert surfaces,
the multiplicity of $B[j]$ on $H_{61}\smallmat{n_o}{r_o/2}{r_o/2}{61m_o}$
being $\sum_{i\ge1}\fcJ{i^2n_om_o,ir_o}{\psi_j}$.  
The divisors of $B[1], B[2]$, and~$B[6]$ in Figure~\ref{table1} show
that $\frak f=B[1]B[6]/B[2]$ is also holomorphic because the Humbert
multiplicities in its divisor are all nonnegative, and so
${\idealf}\in\MFsthreelev{61}^{-}$.
In fact ${\idealf} \in  \CFsthreelev{61}$ 
because the divisor of $\idealf$ contains $H_{61}(1,1)$ and $N=61$ is prime, 
as the following lemma shows.  

\newcommand{\Humhat}{{\hat H}}
\begin{lemma}
\label{lemmahumbert}
Let $N$ be prime and let $f \in \MFskN^\pm$ be a Fricke eigenform.
If~$f$ vanishes on the Humbert surface~$H_{N}(1,1)$ then~$f$ is a cusp form.  
\end{lemma}

\begin{proof}
The Humbert surface $H_N(1,1)$ is $\KN^+\Humhat_N(t_o)$
where $t_o=\smallmat0{1/2}{1/2}N$ and
$\Humhat_N(t_o)=\{\Omega\in\UHPtwo:\ip{\Omega}{t_o}=0\}$.
Define $t[u]=u'tu$ for compatibly sized matrices $t$ and~$u$.
To parametrize $\Humhat_N(t_o)$, choose $\alpha\in\SLtwoR$ having
columns $v_1$ and~$v_2$ such that $t_o[v_1]=t_o[v_2]=0$;
we may take either of $\alpha=\smallmat{\phantom{-}N}1{-1}0,
\smallmat1{-N}0{\phantom{-}1}$.
The parametrization is
\begin{align*}
W_\alpha:\UHP\times\UHP &\lra \Humhat_N(t_o),
\qquad
(\tau,\omega)\longmapsto \tau v_1v_1'+\omega
v_2v_2'=\alpha\smallmat{\tau}00{\omega}\alpha'.
\end{align*}
Recall that $u(\alpha)=\smallmat\alpha00{\alpha^*}$.
The parametrization shows that because $f$ vanishes on
$\Humhat_N(t_o)$, also $f\wtk u(\alpha)\smallmat{\tau}00{\omega}$
vanishes on~$\UHP\times\UHP$ and so clearly $\Phi(f\wtk{u(\alpha)})(\tau)
=\lim_{\lambda\to+\infty}(f\wtk{u(\alpha)})\smallmat{\tau}00{i\lambda}$
is~$0$ for all~$\tau\in\UHP$.

As in section~\ref{sectionParaDefs}, the cusp form condition for~$f$
is that $\Phi(f\wtk g)=0$ for one representative~$g$ of each double coset
$\KN u(\alpha_m)\PtwooneQ$ where $0<m\mid N$ and $\alpha_m=\smallmat1m01$.
Because $N$ is prime here, the only cases are~$m=1,N$,
with representatives $g=u(\alpha)$ for
$\alpha=\smallmat{\phantom{-}N}1{-1}0,\smallmat1{-N}0{\phantom{-}1}$
from the previous paragraph;
the former membership holds because
$\alpha=\smallmat{N+1}{-N}{-1}{\phantom{-}1}
\smallmat1{1}01\smallmat{\phantom{-}1}{0}{-1}1$
and $u$ is a homomorphism,
the latter because $u\smallmat1{2N}01\in\KN$.
The previous paragraph has shown that $\Phi(f\wtk u(\alpha))=0$ for
these two~$\alpha$, so the proof is complete.
\end{proof}

The Fourier coefficients~\cite{yuen22} of $G[1],\dotsc,G[6]$,
and~$\idealf$ show them to span $\CFsthreelev{61}$, so that
$\idealf$ is a nonlift.
Further, the Fourier coefficients show that each element of 
$\CFsthreelev{61}$ is determined by its first and second
Fourier--Jacobi coefficients, so Corollary~\ref{bpecor} gives the
{\em equality\/} of each $B[j]=\BL\left( \psi_j \right)$ and
$G[j]=\Grit\left( \TB_3(D_j) \right)$, and the Borcherds products are
Gritsenko lifts.

\begin{figure}[htb]
$$
\begin{array}{|c||c|c|c|c|}
\hline
\,  & B[1]                   & B[2]  & B[6]      &{\ \idealf=B[1]B[6]/B[2]\ } \\
\hline\hline
H_{61}(1,1)          & 9                        & 9  & 9 & 9 \\
\hline
H_{61}(4,2)        & 3 & 7         & 7     &3            \\
\hline
H_{61}(5,35)        & 7 & 7         & 2     &2            \\
\hline
H_{61}(9,3)        & 4 & 1         & 1     &4            \\
\hline
H_{61}(13,47)        & 3 & 0         & 0     &3            \\
\hline
H_{61}(16,4)        & 0 & 3         & 3     &0            \\
\hline
H_{61}(20,52)        & 4 & 3         & 1     &2            \\
\hline
H_{61}(25,5)        & 1 & 0         & 1     &2            \\
\hline
H_{61}(36,6)        & 0 & 0         & 1     &1            \\
\hline
H_{61}(41,23)        & 0 & 0         & 1     &1            \\
\hline
H_{61}(49,7)        & 1 & 1         & 0     &0            \\
\hline
H_{61}(56,42)        & 1 & 0         & 0     &1            \\
\hline
H_{61}(65,59)        & 0 & 1         & 1     &0            \\
\hline
\end{array}
$$
\caption{Divisors of $B[1],B[2],B[6]$, and~$\idealf$ in
  $\CFsthreelev{61}$\label{table1}}
\end{figure}

With $\CFsthreelev{61}$ spanned, it is a matter of linear algebra to
compute the action of~$\Ttwo$ on the basis, assuming we can compute
enough Fourier coefficients, and then to obtain the nonlift eigenform
$f_{61}=-9G[1]-2G[2]+22G[3]+9G[4]-10G[5]+19G[6]-43{\idealf}$,
a rational function of Gritsenko lifts.
By the product expansion in Theorem~\ref{BPthm}, a Borcherds
product~$\BL(\psi)$ has integral Fourier coefficients when~$\psi$
does, and noting that $\idealf$ is also a Borcherds product because
Borcherds product formation is an additive-to-multiplicative
homomorphism, $f_{61}$ has integral Fourier coefficients.
Also $f_{61}$ has unit content, as shown by the Fourier
coefficients $\fc{\mymat{1}{5}{5}{61}}{f_{61}}=-75$ and
$\fc{\mymat{1}{6}{6}{61}}{f_{61}}=107$.
This completes the new proof of Proposition~\ref{prop61}.
\smallskip

The same argument works for $N=73$ and~$79$;
we only present the key elements.
Here $\dim\CFsthreelev{73}=9$ and the lift space dimension is
$\dim\Jcusp{3}{73}=8$,
and $\dim\CFsthreelev{79}=8$ and the lift space dimension is
$\dim\Jcusp{3}{79}=7$.

\begin{proposition}
\label{prop73}
There is a nonlift Hecke eigenform 
$f_{73} \in \CFsthreelev{73}^{-}(\Z)$ 
with unit content, 
given by a rational function of Gritsenko lifts 
\begin{align*}
f_{73} &= 9 G[1] + 19 G[2] + 2 G[3] - 13 G[4] + 34 G[5] \\
&\quad - 15 G[6] - 12 G[7] - 10 G[8] - 39 \frac{G[2] G[6]}{G[4]}.
\end{align*}
Here $G[j]=\Grit\left(\TB_3(E_j)\right)$ for $j=1,\dotsc,8$ are
Gritsenko lifts of theta blocks, with
\begin{alignat*}2
E_1&=[ 2, 3, 3, 3, 3, 4, 4, 5, 7 ]\qquad
&E_2&=[ 2, 3, 3, 3, 3, 3, 5, 6, 6 ]\\
E_3&=[ 2, 2, 3, 4, 4, 4, 4, 4, 7 ]\qquad
&E_4&=[ 2, 2, 3, 3, 4, 4, 4, 6, 6 ]\\
E_5&=[ 2, 2, 3, 3, 3, 5, 5, 5, 6 ]\qquad
&E_6&=[ 2, 2, 2, 4, 4, 4, 5, 5, 6 ]\\
E_7&=[ 2, 2, 2, 2, 3, 4, 4, 5, 8 ]\qquad
&E_8&=[ 2, 2, 2, 2, 2, 4, 5, 6, 7 ].
\end{alignat*}
The set $\{f_{73},G[1],\dotsc,G[8]\}$ is a basis of~$\CFsthreelev{73}$.
Each Gritsenko lift $G[j]$ is also a Borcherds product
$B[j]=\BL(-(\TB_3(E_j) \wtop{V_2})/\TB_3(E_j))$. %for~$j=1,\dotsc,8$.
\end{proposition}

\begin{proof}
The proof follows the pattern of the proof for $N=61$ in
Proposition~\ref{prop61}.
The Fourier coefficients  $\fc{\mymat{1}{13/2}{13/2}{73}}{f_{73}}=7 $ 
and $\fc{\mymat{1}{-7}{-7}{73}}{f_{73}}=-6 $
prove unit content, and see Figure~\ref{table2} for the 
holomorphy of~$f_{73}$.  
\end{proof}

\begin{figure}[htb]
$$
\begin{array}{|c||c|c|c|c|}
\hline
\,  & B[2]                   & B[4]  & B[6]      &B[2]B[6]/B[4]    \\
\hline\hline
H_{73}(1,1)          & 9                        & 9  & 9 & 9 \\
\hline
H_{73}(4,2)        & 3 & 7         & 7     &3            \\
\hline
H_{73}(8,64)        & 1 & 3         & 3     &1            \\
\hline
H_{73}(9,3)        & 7 & 4         & 1     &4            \\
\hline
H_{73}(12,42)        & 0 & 1         & 1     &0            \\
\hline
H_{73}(16,4)        & 0 & 3         & 3     &0            \\
\hline
H_{73}(24,30)        & 0 & 1         & 3     &2            \\
\hline
H_{73}(25,5)        & 1 & 0         & 2     &3            \\
\hline
H_{73}(36,6)        & 2 & 2         & 1     &1            \\
\hline
H_{73}(37,57)        & 2 & 0         & 0     &2            \\
\hline
H_{73}(48,62)        & 0 & 0         & 1     &1            \\
\hline
H_{73}(65,49)        & 1 & 0         & 0     &1            \\
\hline
H_{73}(72,46)        & 1 & 1         & 0     &0            \\
\hline
H_{73}(73,73)        & 0 & 2         & 2     &0            \\
\hline
\end{array}
$$
\caption{Divisors of $B[2],B[4],B[6]$, and~$B[2]B[6]/B[4]$ in
  $\CFsthreelev{73}$\label{table2}}
\end{figure}

The next proposition gives a simpler expression for~$f_{79}$ than was
given in~\cite{py15}.  This simpler expression makes the congruence
modulo~$32$ visible.

\begin{proposition}
\label{prop79}
There is a nonlift Hecke eigenform 
$f_{79} \in \CFsthreelev{79}^{-}(\Z)$ 
with unit content,
given by a rational function of Gritsenko lifts 
\begin{align*}
f_{79} &= 4 G[1] + 13 G[2] -15 G[3] +8 G[4] + 5 G[6] - 11 G[7]
- 32 \frac{G[2] G[3]}{G[1]}.
\end{align*}
Here $G[j]=\Grit\left(\TB_3(F_j)\right)$ for $j=1,\dotsc,7$ are
Gritsenko lifts of theta blocks, with
\begin{alignat*}2
F_1&=[ 1, 2, 2, 2, 2, 3, 4, 4, 10 ]\qquad
&F_2&=[ 2, 2, 2, 2, 4, 4, 5, 6, 7 ]\\
F_3&=[ 1, 1, 1, 1, 2, 3, 4, 5, 10 ]\qquad
&F_4&=[ 2, 2, 2, 2, 2, 4, 4, 5, 9 ]\\
F_5&=[ 1, 3, 3, 3, 3, 4, 4, 5, 8 ]\qquad
&F_6&=[ 1, 1, 1, 1, 1, 2, 2, 8, 9 ]\\
F_7&=[ 1, 2, 2, 3, 3, 3, 4, 5, 9 ].
\end{alignat*}
The set $\{f_{79},G[1],\dotsc,G[7]\}$ is a basis of~$\CFsthreelev{79}$.
Each Gritsenko lift $G[j]$ is also a Borcherds product
$B[j]=\BL(-(\TB_3(F_j) \wtop{V_2})/\TB_3(F_j))$. %for~$j=1,\dotsc,7$.
\end{proposition}

\begin{proof}
Again the proof follows that of Proposition~\ref{prop61}.
The Fourier coefficients  $\fc{\mymat{1}{6}{6}{79}}{f_{79}}=58$ 
and $\fc{\mymat{1}{-7}{-7}{79}}{f_{79}}=101$ 
prove unit content, and see Figure~\ref{table3} for the holomorphy
of~$f_{79}$.
\end{proof}

\begin{figure}[htb]
$$
\begin{array}{|c||c|c|c|c|}
\hline
\,  & B[1]                   & B[2]  & B[3]      &B[2]B[3]/B[1]    \\
\hline\hline
H_{79}(1,1)          & 9                        & 9  & 9 & 9 \\
\hline
H_{79}(4,2)        & 7 & 7         & 3     &3            \\
\hline
H_{79}(5,59)        & 2 & 2         & 2     &2            \\
\hline
H_{79}(8,18)        & 1 & 1         & 1     &1            \\
\hline
H_{79}(9,3)        & 1 & 1         & 1     &1            \\
\hline
H_{79}(13,31)        & 1 & 0         & 1     &0            \\
\hline
H_{79}(16,4)        & 2 & 2         & 1     &1            \\
\hline
H_{79}(20,40)        & 1 & 1         & 0     &0            \\
\hline
H_{79}(21,69)        & 0 & 2         & 0     &2            \\
\hline
H_{79}(25,5)        & 1 & 1         & 2     &2            \\
\hline
H_{79}(36,6)        & 0 & 1         & 0     &1            \\
\hline
H_{79}(40,44)        & 0 & 0         & 1     &1            \\
\hline
H_{79}(44,26)        & 0 & 1         & 0     &1            \\
\hline
H_{79}(45,19)        & 0 & 0         & 1     &1            \\
\hline
H_{79}(49,7)        & 0 & 1         & 0     &1            \\
\hline
H_{79}(65,67)        & 1 & 1         & 0     &0            \\
\hline
H_{79}(76,32)        & 0 & 0         & 1     &1            \\
\hline
H_{79}(80,78)        & 1 & 1         & 0     &0            \\
\hline
H_{79}(100,10)        & 1 & 0         & 1     &0            \\
\hline
\end{array}
$$
\caption{Divisors of $B[1],B[2],B[3]$, and~$B[2]B[3]/B[1]$ in
  $\CFsthreelev{79}$\label{table3}}
\end{figure}

%%%%%%%%%%%%%%%%%%%%%%%%%%%%%%%%%%
%%%%%%%%%%%%%%%%%%%%%%%%%%%%%%%%%%
%%%%%%%%%%%%%%%%%%%%%%%%%%%%%%%%%%
%%%%%%%%%%%%%%%%%%%%%%%%%%%%%%%%%%
\section{Classification of congruences to Gritsenko
  lifts\label{sectioncongruence}}
%%%%%%%%%%%%%%%%%%%%%%%%%%%%%%%%%%
%%%%%%%%%%%%%%%%%%%%%%%%%%%%%%%%%%
%%%%%%%%%%%%%%%%%%%%%%%%%%%%%%%%%%
%%%%%%%%%%%%%%%%%%%%%%%%%%%%%%%%%%

A few preliminary results, even if very easy or well known, will be
handy for proving Fourier coefficientwise congruences of eigenforms.

\begin{lemma}\label{lemmaone}
Let $p$ be a rational prime.  Let $K$ be an algebraic number field.
Consider an element~$a$ of~$\OK$ and the $\OK$-ideal $\ideala=\id{p,a}$.
If $p\mid\No(a)$ then $\ideala$ is divisible by a prime ideal
over~$p$, and in particular $\ideala$ is not all of~$\OK$.
If $p\mmid\No(a)$ then $\ideala$ is prime and $\No(\ideala)=p$.
\end{lemma}

\begin{proof}
The condition $p\mid\No(a)$ says that a prime ideal~$\idealp$
over~$p$ divides~$\id a$, and also $\idealp$ divides~$p\OK$,
so it divides their sum~$\ideala$ and we have $\No(\ideala)>1$.
The containments $p\OK\subseteq\ideala$ and $\id a\subseteq\ideala$
give the divisibilities $\No(\ideala)\mid\No(p\OK)=p^{[K:\Q]}$ and
$\No(\ideala)\mid\No(\id a)=|\No(a)|$.
So if $p\mmid\No(a)$ then $\No(\ideala)=p$, making $\ideala$ prime.
\end{proof}

\begin{lemma}
\label{lemmatwo}
Let $k,\ell\in\Zpos$ be coprime.
Let $K$ be an algebraic number field, $\idealuu\subseteq\OK$ an ideal,
and $c\in\OK$ an element.
If $\ell c\in\idealuu$ then $c\in\idealuu+k\OK$.
\end{lemma}

\begin{proof}
There exist $m,n\in\Z$ with $1=m\ell+kn$,
so $c=m\ell c+knc\in\idealuu+k \OK$.  
\end{proof}

The following result is the well known theorem of Kummer--Dedekind,
but we state the particular version that we need.

\begin{lemma}\label{mirrorfactorlemma}
Let $K=\Q(a)$ be an algebraic number field with $a\in\OK$.
Let $\varphi(x)\in\Z[x]$ be the minimal polynomial of~$a$.
Let $p$ be a rational prime.
Suppose that $p^2\nmid\disc(\varphi)$ in~$\Z$;
that is, either $p\nmid\disc(\varphi)$ or $p\mmid\disc(\varphi)$.
Then the factorization of $\varphi$ modulo~$p$ into powers of
irreducible polynomials,
\begin{alignat*}2
\varphi(x)&\equiv\prod_{i=1}^g\varphi_i(x)^{e_i}\bmod p\Z[x],
&&\qquad\deg(\varphi_i)=f_i\text{ for each~$i$},\\
\intertext{gives the factorization of~$p$ in~$\OK$ into powers of
  prime ideals,}
p\OK&=\prod_{i=1}^g\id{p,\varphi_i(a)}^{e_i},
&&\qquad\operatorname{N}(\id{p,\varphi_i(a)})=p^{f_i}\text{ for each~$i$}.
\end{alignat*}
\end{lemma}

\begin{proof}
Because $\disc(\varphi)=\disc(K)\cdot|\OK/\Z[a]|^2$, the hypotheses
give $p\nmid|\OK/\Z[a]|$, and so Theorem~27 of~\cite{Marcus2018}
gives the result.
\end{proof}

The next lemma gives the action of the standard Hecke operators on
Fourier coefficients when the bad prime divides the level at most once.

\begin{lemma}
\label{lemmafour}
Let $f\in\CFskN$ with $k\ge3$ and $N\in\Zpos$.
Let~$p\in\Zpos$ be prime with $p\nmid N$ or $p\mmid N$.
If $p\nmid N$ then let $\HeckeT$ be either of $\Tp,\Topsq$,
and if $p\mmid N$ then let $\HeckeT$ be either of $\Tp,\Tozpsq$.
Then the Fourier coefficients of~$f \wtk{\HeckeT}$ are fixed
$\Z$-linear combinations of the Fourier coefficients of~$f$,
independent of~$f$.
\end{lemma}

\begin{proof}
For $p\nmid N$ and $\HeckeT=\Tp,\Topsq$, see \cite{bpptvy18}, pp.1165--1168.

For $p\mmid N$ and $\HeckeT=\Tp,\Tozpsq$, the result follows from the
following formulas, which hold for~$k\ge0$ and clearly have integral
coefficients for~$k\ge3$.
Let $M \equiv (N/p)^{-1} \bmod p$ and let $a,c\in\Z$ be such that $ap-cN/p=1$,
so that $NM/p\equiv1\bmod p$ and $ap^2-cN=p$.
Then for any~$t\in\XtwoN$,
recalling that $t[u]=u'tu$ for compatibly sized matrices $t$ and~$u$,
\begin{align*}
\fc t{f|\Tp}
&=\fc{pt}f
+p^{k-2}
\sum_{x\bmod p}\fc{\tfrac1p\,t\left[\smallmat{\phantom{-}1}0{-x}p\right]}f \\
&\quad+p^{k-2} \sum_{y\bmod p}\fc{\tfrac1p\,t\left[\smallmat{p}{NMy}{0}{1}\right]}f
+p^{2k-3}\fc{\tfrac1p\,t}f \\
&\quad+p^{k-3}
\left\{\begin{alignedat}2
p-1&\quad&&\text{if } p \mid 2 t_{12} \\
-1&\quad&&\text{else} 
\end{alignedat}\right\}
\fc{\tfrac1p\,t\left[\smallmat{ap}{N}{c}{p}\right]}f,
\end{align*}
and
\begin{align*}
\fc t{f|\Tozpsq}
&=p^{k-3} \sum_{x\bmod p}\fc{t\left[\smallmat{1}{0}{-x}{p}\right]}f \\
&\quad+p^{3k-6}\sum_{y\bmod p}\fc{t\left[\smallmat{1}{NMy/p}{0}{1/p}\right]}f \\
&\quad+p^{2k-6}\sum_{y\bmod p} 
\left\{\begin{alignedat}2
p-1&\quad&&\text{if }p\,\bigg|\,
\begin{aligned}
  &2 t_{12}(1 + 2c N/p + 2cy) \\
  &\quad+2t_{22}c/p
\end{aligned} \\
-1&\quad&&\text{else} 
\end{alignedat}\right\} \\
&\qquad\quad\cdot\fc{t\left[\smallmat{(cN+p+cNMy)/p}{N+NMy}{c/p}{1}\right]}f \\
&\quad+p^{2k-6}\sum_{x,y\bmod p}
\left\{\begin{alignedat}2
p-1&\quad&&\text{if $p \mid 2 t_{12}My+t_{22}/N$} \\
-1&\quad&&\text{else} 
\end{alignedat}\right\} \\
&\qquad\quad\cdot\fc{t\left[\smallmat{1+NMxy/p}{NMy}{x/p}{1}\right]}f.
\end{align*}

\newcommand\xhat{{\hat x}}
These formulas are derived from the double coset formulas for $\Tp$
and~$\Tozpsq$ when $p\mmid N$, as given in~\cite{schmidt18}.
The derivation relies on the following result:
with $p,N,M,a,c$ as above, also consider any $x\ne0\bmod p$,
and let $\xhat=x^{-1}\bmod p$, so that $x\xhat=1\bmod p$.
Let $S\in\KN$ be the matrix  
$$
S=\left(
\begin{array}{cccc}
ap   &      N   &      N \xhat/p   &     a \xhat \\
c      &     p   &        \xhat      &       c \xhat/p \\
cNMx/p & NMx  & 1+NMx \xhat/p & c(NMx \xhat/p-1)/p \\
aNMx & N^2Mx/p & N(NMx \xhat/p-1)/p &  a+aNMx \xhat/p\\
\end{array}
\right)
$$
and let $U$ be the blockwise upper triangular matrix
$$
U=\left(
\begin{array}{cccc}
1  &  -N/p    &    -N \xhat/p   &  -\xhat \\
-c/p   &  a   &     -a \xhat   &    -c \xhat/p \\
0    &  0    &     a p      &        c \\
0    &  0    &      N        &       p \\
\end{array}
\right).
$$
Then their product is the $0$-dimensional cusp $C_0(NMx/p)$ as defined
in~\cite{py13} p.449.
$$
SU=\left(
\begin{array}{cccc}
1   &      0   &      0   &    0\\
0    &     1   &      0   &    0\\
0   &   NMx/p    &    1   &    0\\
NMx/p   &    0   &    0   &    1
\end{array}
\right).
$$
From here we can upper triangularize all the matrices in the formulas
from~\cite{schmidt18}, and we get the above two Fourier coefficient
formulas.
The formulas from~\cite{schmidt18} give the analytic, scalar invariant
normalization of the slash operator.  So, beyond upper triangularizing,
the first formula from~\cite{schmidt18} has to be multiplied
by~$p^{k-3}$ and the second by~$p^{2k-6}$ for the arithmetic normalization.
\end{proof}

In the following lemma the ideal~$\ideala$ does not need to be prime.  

\begin{lemma}
\label{lemmathree}
Let $k\ge3$ and $N$ be positive integers.
Let $p\in\Zpos$ be prime with $p\nmid N$ or $p\mmid N$.
If $p\nmid N$ then let $\HeckeT$ be either of $\Tp,\Topsq$,
and if $p\mmid N$ then let $\HeckeT$ be either of $\Tp,\Tozpsq$.
Let $K$ be an algebraic number field.
Consider two $\HeckeT$-eigenforms lying in the same Fricke space, and
having the same Atkin--Lehner sign~$\epsilon_p$ if~$p\mmid N$,
these being
$f\in\CFskN^\pm(\Z)$, having unit content, and $g\in\CFskN^\pm(\OK)$.
Let $\ideala \subseteq \OK$ be an ideal, and assume that $f$ and~$g$
have congruent Fourier coefficients modulo~$\ideala$,
$$
\fc tf \equiv \fc tg \bmod \ideala\quad\text{for all $t\in\XtwoN$}.
$$
Then $f$ and~$g$ have congruent $\HeckeT$-eigenvalues modulo~$\ideala$,
$$
\evhoef\HeckeT f\equiv \evhoef\HeckeT g \bmod \ideala.
$$
If
$p\nmid N$ and $f$ and~$g$ are $\HeckeT$-eigenforms for both
$\HeckeT=\Tp$ and $T=\Topsq$,
or if $p\mmid N$ and $f$ and~$g$ are $\HeckeT$-eigenforms for both
$\HeckeT=\Tp$ and $T=\Tozpsq$,
then $f$ and~$g$ have congruent $p$-Euler polynomials modulo~$\ideala$.
\end{lemma}

\begin{proof}
By Lemma~\ref{lemmafour}, each
$\fc{t_o}{f\wtk{\HeckeT}}=\evhoef\HeckeT f\fc{t_o}f$ is a
$\Z$-linear combination of the Fourier coefficients~$\fc tf$.
Thus $\evhoef\HeckeT f$ lies in~$\Q$.
Also the characteristic polynomial of~$\HeckeT$ is monic in~$\Z[x]$,
making its roots algebraic integers, so in fact $\evhoef\HeckeT f$
lies in~$\Z$.  Similarly $\evhoef\HeckeT g$ lies in~$\OK$.
For all $t \in \XtwoN$, because $\fc tf$ and~$\fc tg$ are equivalent
modulo~$\ideala$, and because $\fc t{f\wtk{\HeckeT}}$ and
$\fc t{g\wtk{\HeckeT}}$ are the same $\Z$-linear combination of the
$\fc uf$ and the $\fc ug$, they are equivalent modulo~$\ideala$ as
well.  So, because $f$ and~$g$ are eigenforms,
$(\evhoef\HeckeT f-\evhoef\HeckeT g)\fc tf
=\fc t{f\wtk{\HeckeT}}-\fc t{g\wtk{\HeckeT}}$ lies in~$\ideala$.
Because $f$ has unit content some finite $\Z$-linear combination
$\sum_tn_t\fc tf$ is~$1$, and so the eigenvalue difference
$\evhoef\HeckeT f-\evhoef\HeckeT g
=\sum_tn_t(\evhoef\HeckeT f-\evhoef\HeckeT g)\fc tf$
lies in~$\ideala$.
Thus $\evhoef\HeckeT f\equiv\evhoef\HeckeT g\bmod\ideala$.

As for the $p$-Euler polynomials of $f$ and~$g$, for $p\nmid N$
let $f\wtk{\Tp}=\evhoef\Tp ff$ and $f\wtk{\Topsq}=\evhoef\Topsq ff$.
Then the $p$-Euler polynomial is given in~(4.2.16) of~\cite{bpptvy18},
\begin{align*}
Q_p(f,x)&=1-\evhoef\Tp fx+(p\evhoef\Topsq f+p^{2k-5}(1+p^2))x^2\\
&\qquad-p^{2k-3}\evhoef\Tp fx^3+p^{4k-6}x^4,
\end{align*}
and this is determined modulo~$\ideala$ by $\evhoef\Tp f$
and~$\evhoef\Topsq f$ modulo~$\ideala$.
For a prime $p\mmid N$
let $f\wtk{\Tp}=\evhoef\Tp ff$ and $f\wtk{\Tozpsq}=\evhoef\Tozpsq ff$,
and recall that $\epsilon_p$ denotes the shared Atkin--Lehner sign
of $f$ and~$g$.
In this case the $p$-Euler polynomial is, see~\cite{MR2887605}, pg. 547,  
\begin{align*}
Q_p(f,x)&=1-(\evhoef\Tp f+p^{k-3} \epsilon_p)x
+(p\evhoef\Tozpsq f+p^{2k-3})x^2+p^{3k-5}\epsilon_px^3,
\end{align*}
and this is determined modulo~$\ideala$ by $\evhoef\Tp f$,
$\evhoef\Tozpsq f$, and~$\epsilon_p$ modulo~$\ideala$.
We remark that this last case can be proved without the formula for
$\fc t{f\wtop{\Tozpsq}}$ in Lemma~\ref{lemmafour},
because for $N=p$ the bad Euler polynomial only depends upon
$\evhoef\Tp f$ and~~$\epsilon_p$.
%both of which are~$-1$.
Indeed, we have the relation 
$p^{k-3} \evhoef\Tp f\epsilon_p+\evhoef\Tozpsq f+p^{2k-5}+p^{2k-6}=0$.
The reference for this relation for local representations is \cite{robertsschmidt07}, pg. 248.    
\end{proof}

With the needed supporting results in place, we can establish the main
results of this section.

\begin{theorem}\label{cong61}
Let $\GritSp \subseteq \CFsthreelev{61}$ be the subspace of Gritsenko lifts.  
The characteristic polynomial of $\Ttwo$ on~$\GritSp$ is irreducible
over~$\Q$,
$$
q(x)= x^6 -29x^5 +322x^4 -1714x^3 +4471x^2 -5205x +2026.  
$$
Let~$a$ be a root of~$q$ and let $K=\Q(a)$.  
With reference to the elements $d_j(a)$ of~$K$ in Figure~\ref{table4}
and to the Gritsenko lifts $G[j]$ from Proposition~\ref{prop61},
consider an element of~$\GritSp(K)$,
$$
g(a)=\sum_{j=1}^6 d_j(a) G[j].
$$
Then $g(a)$ lies in~$\GritSp(\OK)$,
and it is a $\Ttwo$-eigenform with eigenvalue~$a$,
and it is an eigenform of~$\Tp$ and~$\Tozpsq$ for all primes~$p$.
The $\OK$-ideal
$$
\ideala=\id{43,a+7}
$$
is prime.  The Fourier coefficients and the Euler polynomials of
$f_{61}$ and~$g(a)$ are congruent modulo~$\ideala$.
The ideal~$\ideala$ is the only (proper) $\OK$-ideal that gives a
congruence between the Euler polynomials of $f_{61}$ and~$g(a)$.
\end{theorem}

\begin{figure}[htb]
$$
\begin{array}{|c||c|c|}
\hline
j\,  & c_j & d_{j\down\!\!\!}(a) \\
\hline\hline
\phantom{\ \ }1\up & -9\phantom{-}\down
& \frac{515}{2}  -\frac{2377}{4} a +443 a^2 -\frac{269}{2} a^3 +17 a^4
-\frac{3}{4} a^5 \\
\hline
\phantom{\ \ }2\up & -2\phantom{-}\down
& \frac{1899}{8}  -\frac{7813}{16} a +\frac{1223}{4} a^2
-\frac{649}{8} a^3 +\frac{39}{4} a^4 -\frac{7}{16} a^5 \\
\hline
\phantom{\ \ }3\up & 22\down 
& -\frac{305}{2}  +\frac{1697}{4} a -359 a^2 +\frac{237}{2} a^3 -16
a^4 +\frac{3}{4} a^5 \\
\hline
\phantom{\ \ }4\up & 9\down 
& -396  +855 a -596 a^2 +174 a^3 -22 a^4 + a^5 \\
\hline
\phantom{\ \ }5\up & -10\phantom{-}\down 
& -140  +215 a -97 a^2 +17 a^3 - a^4 \\
\hline
\phantom{\ \ }6\up & 19\down & -24 \\
\hline
\end{array}
$$
\caption{Coefficients of Gritsenko lifts for $N=61$\label{table4}}
\end{figure}

\begin{proof}
The Gritsenko lift space $\GritSp$ has basis $\{G[j]:j=1,\dotsc,6\}$.
We can compute enough Fourier coefficients of the~$G[j]$ (see~\cite{yuen22})
to give the matrix of~$\Ttwo$ on~$\GritSp$ for the basis, acting
from the left on column vectors of~$\C^6$,
$$
M=\left(\begin{array}{rrrrrr}
 9 & 0 & 4 & -1 & 4 & 1 \\
 -6 & 1 & -4 & -4 & 0 & 0 \\
 3 & 0 & 8 & 2 & 2 & 0 \\
 3 & 4 & 0 & 11 & -4 & -2 \\
 -8 & 0 & -8 & -1 & -3 & -1 \\
 -1 & -4 & 2 & -4 & 0 & 3 \\
\end{array}\right).
$$
The characteristic polynomial~$q(x)=\det(xI-M)$ is as stated.

Each $d_j(a)$ lies in~$\OK$, as can be confirmed by computer
software, and each $G[j]$ lies in $\GritSp(\Z)$, so $g(a)$ lies
in~$\GritSp(\OK)$.
The vector $(d_1(a), \dotsc, d_6(a))$, viewed as a column, lies in
$\operatorname{null}\left(aI-M\right)$, and so $g(a)$ is a
$\Ttwo$-eigenform in~$\GritSp(\OK)$ with eigenvalue~$a$.
The characteristic polynomial of~$\Ttwo$ is separable because its
discriminant $\disc(q)=2^{14}\cdot3^6\cdot1892022169$ is nonzero.
So the eigenspaces of~$\Ttwo$ are one-dimensional, and consequently
$g(a)$ is an eigenform of every Hecke operator that commutes with~$\Ttwo$.
The commutativity is automatic for $\Tp$ and $\Tozpsq$ with~$p\ne2$.
For~$\HToz4$, the commutator~$[\Ttwo,\HToz4]$ consists of level
lowering operators by Proposition~{6.21} in \cite{robertsschmidt07}
and so $[\Ttwo,\HToz4]=0$ on $\CFsthreelev{61}$ because $61$ is prime
and $\CFsthreelev{1}=\{0\}$.
Thus $g(a)$ is an eigenform of all the stated Hecke operators.  

The reduction of~$q(x)$ modulo~$43$ is
$$
q(x)\equiv(x+7)(x+25)(x+30)(x^3+38x^2+13x+12)\bmod43,
$$
and because $43\nmid\disc(q)$,
Lemma~\ref{mirrorfactorlemma} says that
the $\OK$-ideal $\ideala=\id{43,a+7}$ is prime,
and more generally that $\ideala$, $\id{43,a+25}$, and $\id{43,a+30}$ are
the norm-$43$ ideals of~$\OK$.

The values $c_j$ in Figure~\ref{table4} are such that
$f_{61} = 43 {\idealf} + \sum_{j=1}^6 c_j G[j]$ in
Proposition~\ref{prop61}.  Thus
$$
f_{61}-g(a) = 43 {\idealf} + \sum_{j=1}^6 (c_j-d_j(a)) G[j].
$$
%From just above, $\No(a+7)=2223616=2^9 \cdot 43 \cdot 101$.
For $j=1, \dotsc, 6$ we compute that $1616(c_j-d_j(a))\in\id{a+7}$.  
Lemma~\ref{lemmatwo} with $k=43$, $\ell=1616$, $\idealuu=\id{a+7}$,
and $c=c_j-d_j(a)$ gives $c_j-d_j(a)\in\id{43,a+7}=\ideala$.
The Fourier coefficients of~$43\idealf$ lie in~$\ideala$ as well.
Thus $f_{61} \equiv g(a) \bmod \ideala$ at the level of Fourier coefficients.  
Because $\CFsthreelev{61}=\CFsthreelev{61}^-$ and $61$ is prime,
Lemma~\ref{lemmathree} says that
all $p$-Euler polynomials of $f_{61}$ and~$g(a)$
are congruent modulo~$\ideala$.

The matrix of~$\Tthree$ on~$\GritSp$ is given in~\cite{yuen22}.
Using it along with the matrix of~$\Ttwo$
we compute two eigenvalue differences,
\begin{align*}
  \evhoef\Ttwo{g(a)}-\evhoef\Ttwo{f_{61}}
  &=a+7 \\
  \evhoef\Tthree{g(a)}-\evhoef\Tthree{f_{61}}
  &=\tfrac{1495}{48} - \tfrac{4313}{96} a+
  \tfrac{691}{24}a^2 - \tfrac{349}{48} a^3+ \tfrac{19}{24} a^4 -
  \tfrac{1}{32} a^5 + 3.
\end{align*}
The norm of any ideal containing these eigenvalue differences
divides their norms,
\begin{align*}
\No\left(\evhoef\Ttwo{g(a)} - \evhoef\Ttwo{f_{61}}\right)
&=2^9 \cdot 43 \cdot 101 \\
\No\left(\evhoef\Tthree{g(a)} - \evhoef\Tthree{f_{61}}\right)
&=5^2 \cdot 19 \cdot 43  \cdot 139,
\end{align*}
and therefore divides the greatest common divisor $43$ of these norms,
and therefore, because the ideal is proper, equals~$43$.  So the ideal
is one of the norm-$43$ ideals $\ideala=\id{43,a+7},\id{43,a+25},\id{43,a+30}$,
and it contains~$a+7$, so it is~$\ideala$.
\end{proof}

The elements $d_j(a)$ of~$\OK$ in Figure~\ref{table4} were determined as follows.
Select some $v\in K^6\cap\ker(M-aI)$.
Find the minimal $\ell\in\Zpos$ such that $\ell v\in\OK^6$.
Search for $b\in\Z$ such that $43\mid\No(c_j-b\ell v_j)$ for $j=1,\dotsc,6$,
with $c_j$ from Figure~\ref{table4}.  Then $b\ell v_j$ has the
properties that we want, and as an element of~$\OK$ it takes the form
$d_j(a)$.

The congruence of eigenvalues modulo a prime ideal above~$43$ is also
proven in~\cite{dummiganprt2021}, which also contains a proof of a
congruence above~$19$, discovered by Buzzard and Golyshev, of~$f_{61}$
to a Yoshida lift.

%%   gamma*v = {632307/4 - (8480021 L)/16 + (36446543 L^2)/64 - (16486531 L^3)/64 + (
%%    442583 L^4)/8 - (357943 L^5)/64 + (13671 L^6)/64, 
%%   247387/2 - (2906781 L)/8 + (9639367 L^2)/32 - (3516515 L^3)/32 + (
%%    77947 L^4)/4 - (52719 L^5)/32 + (1703 L^6)/32, -(134939/4) + (
%%    1626701 L)/16 - (6689239 L^2)/64 + (2996715 L^3)/64 - (80271 L^4)/
%%    8 + (64895 L^5)/64 - (2479 L^6)/64, 
%%   280607/2 - (3673513 L)/8 + (15573571 L^2)/32 - (6924839 L^3)/32 + (
%%    181543 L^4)/4 - (142859 L^5)/32 + (5307 L^6)/32, -(411139/4) + (
%%    4124773 L)/16 - (16534783 L^2)/64 + (7232243 L^3)/64 - (187815 L^4)/
%%    8 + (146535 L^5)/64 - (5399 L^6)/64, -(850479/4) + (12504265 L)/
%%    16 - (51374171 L^2)/64 + (22172735 L^3)/64 - (569099 L^4)/8 + (
%%    441443 L^5)/64 - (16243 L^6)/64, -(648207/4) + (8826201 L)/16 - (
%%    37654939 L^2)/64 + (16649519 L^3)/64 - (434027 L^4)/8 + (
%%    340515 L^5)/64 - (12643 L^6)/64, 10072}

\begin{theorem}\label{cong73}
Let $\GritSp \subseteq \CFsthreelev{73}$ be the subspace of Gritsenko lifts.  
The characteristic polynomial~$q$ of $\Ttwo$ on~$\GritSp$ factors
over~$\Q$ as $q=q_1q_7$, where its irreducible factors $q_1$ and~$q_7$
are
\begin{align*}
q_1(x)  &= x-9,  \\
q_7(x)  &= x^7 -30 x^6 + 357 x^5 -2157 x^4 +7034 x^3 -12145 x^2 + 9964 x -2832.  
\end{align*}

With reference to the integers $d_{1,j}$ in Figure~\ref{table5}
and to the Gritsenko lifts $G[j]$ from Proposition~\ref{prop73},
consider an element of~$\GritSp(\Z)$,
$$
g_1= \sum_{j=1}^8 d_{1,j} G[j].
$$
Then $g_1$ is a $\Ttwo$-eigenform with eigenvalue~$9$,
and it is an eigenform of~$\Tp$ and~$\Tozpsq$ for all primes~$p$.
The Fourier coefficients and the Euler polynomials of $f_{73}$ and~$g_1$ 
are congruent modulo~$3\Z$.  The only (proper) $\Z$-ideal that gives a
congruence between the Euler polynomials of $f_{73}$ and~$g_1$ is~$3\Z$.

Let~$a$ be a root of~$q_7$ and~$K=\Q(a)$.  
With reference to the elements $d_{7,j}(a)$ of~$K$ in Figure~\ref{table5}
and to the Gritsenko lifts $G[j]$ from Proposition~\ref{prop73},
consider an element of~$\GritSp(K)$,
$$
g_7(a)= \sum_{j=1}^8 d_{7,j}(a) G[j].
$$
Then $g_7(a)$ lies in~$\GritSp(\OK)$,
and it is a $\Ttwo$-eigenform with eigenvalue~$a$,
and it is an eigenform of~$\Tp$ and~$\Tozpsq$ for all primes~$p$.
The $\OK$-ideals
$$
\ideala=\id{3,a},\qquad\idealb=\id{13,a+6}
$$
are prime.  The Fourier coefficients and the Euler polynomials of
$f_{73}$ and~$g_7(a)$ are congruent modulo~$\ideala$ and modulo~$\idealb$.
The only prime-power $\OK$-ideals that give congruences between the Euler 
polynomials of $f_{73}$ and~$g_7(a)$ are $\ideala$ and~$\idealb$.  
\end{theorem}

\newcommand\mypushup{^{\phantom{'}}}
\newcommand\mypushdn{_{\phantom{'}}}

\begin{figure}[htb]
$$
\begin{array}{|c||r|r|c|}
\hline
j\, & c_j & d_{1,j} & d_{7,j\down\!\!\!}(a)\up \\
\hline\hline
1   & 9   & -3                
& \begin{aligned}
    & -\tfrac{632307\mypushup}{4} + \tfrac{8480021}{16}a
    - \tfrac{36446543}{64}a^2 + \tfrac{16486531}{64}a^3 \\
    & \quad - \tfrac{442583}{8\mypushdn}a^4
    + \tfrac{357943}{64}a^5 - \tfrac{13671}{64}a^6
  \end{aligned} \\
\hline
2   & 19   & -2
& \begin{aligned}
    & -\tfrac{247387\mypushup}{2} + \tfrac{2906781}{8}a
    - \tfrac{9639367}{32}a^2 + \tfrac{3516515}{32}a^3 \\
    & \quad - \tfrac{77947}{4\mypushdn}a^4 + \tfrac{52719}{32}a^5
    - \tfrac{1703}{32}a^6
  \end{aligned} \\
\hline
3   & 2   & -1
& \begin{aligned}
    & \tfrac{134939\mypushup}{4} - \tfrac{1626701}{16}a
    + \tfrac{6689239}{64}a^2 - \tfrac{2996715}{64}a^3 \\
    & \quad + \tfrac{80271}{8\mypushdn}a^4 - \tfrac{64895}{64}a^5
    + \tfrac{2479}{64}a^6
  \end{aligned} \\    
\hline
4   & -13  & 2
& \begin{aligned}
    & -\tfrac{280607\mypushup}{2} + \tfrac{3673513}{8}a
    - \tfrac{15573571}{32}a^2 + \tfrac{6924839}{32}a^3 \\
    & \quad - \tfrac{181543}{4\mypushdn}a^4 + \tfrac{142859}{32}a^5
    - \tfrac{5307}{32}a^6
  \end{aligned} \\    
\hline
5   & 34   & 1
& \begin{aligned}
    & \tfrac{411139\mypushup}{4} - \tfrac{4124773}{16}a
    + \tfrac{16534783}{64}a^2 - \tfrac{7232243}{64}a^3 \\
    & \quad + \tfrac{187815}{8\mypushdn}a^4 - \tfrac{146535}{64}a^5
    + \tfrac{5399}{64}a^6    
  \end{aligned} \\    
\hline
6   & -15  & 3
& \begin{aligned}
    &  \tfrac{850479\mypushup}{4} - \tfrac{12504265}{16}a
    + \tfrac{51374171}{64}a^2 - \tfrac{22172735}{64}a^3 \\
    & \quad + \tfrac{569099}{8\mypushdn}a^4 - \tfrac{441443}{64}a^5
    + \tfrac{16243}{64}a^6    
  \end{aligned} \\    
\hline
7   & -12 & 3
& \begin{aligned}
    &  \tfrac{648207\mypushup}{4} - \tfrac{8826201}{16}a
    + \tfrac{37654939}{64}a^2 - \tfrac{16649519}{64}a^3 \\
    & \quad + \tfrac{434027}{8\mypushdn}a^4 - \tfrac{340515}{64}a^5
    + \tfrac{12643}{64}a^6    
  \end{aligned} \\    
\hline
8   & -10  & 2  &   -10072 \\
\hline
\end{array}
$$
\caption{Coefficients of Gritsenko lifts for $N=73$\label{table5}}
\end{figure}

\begin{proof}
Similarly to the proof of Theorem~\ref{cong61},
the matrix~$M$ of $\Ttwo$ on~$\GritSp$
for the basis $\{G[j]:j=1,\dotsc,8\}$ is
$$
M=
\left(
\begin{array}{rrrrrrrr}
 8 & 0 & 3 & -4 & -2 & -2 & 4 & 2 \\
 -4 & 1 & -1 & -2 & -3 & 0 & -4 & -5 \\
 0 & 0 & 6 & 0 & -1 & 0 & 0 & -1 \\
 2 & 2 & 5 & 8 & 3 & 2 & 2 & 1 \\
 2 & 4 & -1 & 6 & 8 & 8 & -4 & -5 \\
 -4 & -4 & -8 & -4 & -1 & -4 & 2 & 7 \\
 -2 & 0 & -3 & 4 & 3 & 2 & 1 & -1 \\
 -3 & -4 & -7 & -6 & -4 & -8 & 4 & 11 \\
\end{array}
\right).
$$
The characteristic polynomial~$q(x)=\det(xI-M)=q_1(x)q_7(x)$ is as
stated.

The vector $(d_{1,1},\dotsc,d_{1,8})$ lies in $\operatorname{null}(9I-M)$,
and so $g_1$ is a $\Ttwo$-eigenform with eigenvalue~$9$.   
The characteristic polynomial~$q$ of~$\Ttwo$ is separable because its
discriminant
$\disc(q)=2^{36}\cdot3^3\cdot5^2\cdot13\cdot19^2\cdot37\cdot101\cdot30931$
is nonzero,
and so $g_1$ is an eigenform of all the stated Hecke operators
as in the proof of Theorem~\ref{cong61}.

The values~$c_j$ in Figure~\ref{table5} are such that
$f_{73} = -39 {\idealf} + \sum_{j=1}^8 c_j G[j]$ in
Proposition~\ref{prop73}.  Thus
$$
f_{73}-g_1 = -39 {\idealf} + \sum_{j=1}^8 (c_j-d_{1,j}) G[j].
$$
Figure~\ref{table5} shows that the coefficients $c_j-d_{1,j}$ all lie in~$3\Z$,
as do the Fourier coefficients of~$-39\idealf$,
and so the Fourier coefficients of~$g_1$ and $f_{73}$ are congruent
modulo~$3\Z$.
Because $\CFsthreelev{73}=\CFsthreelev{73}^-$ and $73$ is prime,
Lemma~\ref{lemmathree} says that
all $p$-Euler polynomials of $f_{73}$ and~$g_1$
are congruent modulo~$3\Z$.
Again the matrix of~$\Tthree$ on~$\GritSp$ is given in~\cite{yuen22}.
The eigenvalue differences $\evhoef\Tp{g_1}-\evhoef\Tp{f_{73}}$ for
$p=2,3$ are
\begin{align*}
 \evhoef\Ttwo{g_1} - \evhoef\Ttwo{f_{73}} &= 9+6 = 15 = 3\cdot 5, \\
 \evhoef\Tthree{g_1} - \evhoef\Tthree{f_{73}} &= 4+2 = 6 = 2\cdot 3.
\end{align*}
Any ideal containing these contains~$3\Z$, so it is~$3\Z$.

Now let~$a$ be a root of~$q_7$ and $K=\Q(a)$.    
Each $d_{7,j}(a)$ lies in~$\OK$
%, as can be confirmed by computer software,
and each $G[j]$ lies in $\GritSp(\Z)$,
so $g_7(a)$ lies in~$\GritSp(\OK)$.
The vector $(d_{7,1}(a),\dotsc,d_{7,8}(a))$
lies in $\operatorname{null}(aI-M)$, and so $g_7(a)$ is a
$\Ttwo$-eigenform in~$\GritSp(\OK)$ with eigenvalue~$a$.
Because the roots of~$q_7$ are distinct, 
$g_7(a)$ is an eigenform of all the stated Hecke operators.
%

%% L = lambda = a root of -2832 + 9964 x - 12145 x^2 + 7034 x^3 - 2157 x^4 + 
%%    357 x^5 - 30 x^6 + x^7
%% 
%% [more data requested]
%% lambda_2(Grit) = L
%% lambda_3(Grit) = 1/64 (3856 - 12684 L + 13965 L^2 - 6297 L^3 + 1336
%% L^4 - 133 L^5 +
%%    5 L^6)
%% lambda_5Grit) = 1/8 (-1224 + 5138 L - 5509 L^2 + 2486 L^3 - 530 L^4
%% + 53 L^5 - 2 L^6)
%% lambda_7(Grit) = 1/16 (-3008 + 11572 L - 10489 L^2 + 4215 L^3 - 840
%% L^4 + 81 L^5 -
%%    3 L^6)
%% 
%% lambda_2(Grit) - lambda_2(f)  = L - (-6) has norm that factors into
%% {{2, 2}, {3, 1}, {5, 2}, {13, 1}, {2377, 1}}
%% lambda_3(Grit) - lambda_3(f) has norm that factors into {{2, 2},
%% {3, 1}, {13, 1}, {195809, 1}}
%% 
%% This proves 2, 3, 13 are the only possible congruence primes  

Because
$\operatorname{disc}(q_7)=2^{24}\cdot3\cdot13\cdot19^2\cdot37\cdot101\cdot30931$
is divisible by~$3$ only once, Lemma~\ref{mirrorfactorlemma} says that
the factorization of $q_7$ modulo~$3$,
$$
q_7(x)\equiv x(x+2)^2(x^4+2x^3+x+1)\bmod3,
$$
determines the factorization $3\OK=\id{3,a}\id{3,a+2}^2\id{3,a^4+2a^3+a+1}$,
the first two prime ideals on the right side having norm~$3$ and the
third having norm~$3^4$.
Similarly the factorization of~$q_7$ modulo~$13$,
$$
q_7(x)\equiv(x+6)^2(x^5+10x^4+6x^3+11x^2+4x+8)\bmod{13},
$$
determines the factorization
$13\OK=\id{13,a+6}^2\id{13,a^5+10a^4+6a^3+11a^2+4a+8}$.
Here $\id{3,a}$ and $\id{13,a+6}$ are the ideals~$\ideala$
and~$\idealb$ of the theorem, so those ideals are prime as claimed.

The values $c_j$ in Figure~\ref{table5} are such that
$f_{73} = -39 {\idealf} + \sum_{j=1}^8 c_j G[j]$
in Proposition~\ref{prop73}.  Thus
$$
f_{73}-g_7(a) = -39 {\idealf} + \sum_{j=1}^8 (c_j-d_{7,j}(a)) G[j].
$$
Let
$$
w=\tfrac{134943}{2} - \tfrac{1626701}{8}a + \tfrac{6689239}{32}a^2 
- \tfrac{2996715}{32}a^3 + 
 \tfrac{80271}{4}a^4 - \tfrac{64895}{32}a^5 + \tfrac{2479}{32}a^6. 
$$  
This element of~$\OK$ has norm
\begin{align*}
\No(w)
%&=-25080505128911932253708928 \\
&=-2^7 \cdot 3 \cdot 3919 \cdot 1941571 \cdot 8583739212883,
\end{align*} 
and so by Lemma~\ref{lemmaone}, the ideal $\id{3,w}$ is one of the
norm-$3$ $\OK$-ideals, $\id{3,w}=\ideala=\id{3,a}$ or $\id{3,w}=\id{3,a+2}$.
Further, $32w$ does not lie in~$\id{3,a+2}$, as one can see
by replacing $a$ by~$1$ in~$32w$ and then reducing modulo~$3$.
Therefore $\id{3,w}=\ideala$.
Now set  $\ell=130627630879749647154734$.
For $j=1,\dotsc,8$, compute $\ell(c_j+2d_{7,j}(a))\in\id w$.
Lemma~\ref{lemmatwo} with $k=3$, $\ell$ as given, $\idealuu=\id w$,
and $c=c_j+2d_{7,j}(a)$ gives $c_j+2d_{7,j}(a)\in\id{3,w}=\ideala$,
and it follows that $c_j-d_{7,j}(a)\in\ideala$.
The Fourier coefficients of~$-39\idealf$ lie in~$\ideala$ as well.
Thus $f_{73}\equiv g_7(a)\bmod\ideala$ at the level of Fourier coefficients.
Turning to the $\idealb=\id{13,a+6}$ congruence, we evaluate the norm 
\begin{align*}
\No(a+6) =9270300  
=2^2 \cdot 3 \cdot 5^2 \cdot 13 \cdot 2377.  
\end{align*} 
Set $\ell= 356550$.
%=2 \cdot 3 \cdot 5^2 \cdot 2377$.
For $j=1,\dotsc,8$, compute $\ell(c_j+12d_{7,j}(a))\in\id{a+6}$.  
Lemma~\ref{lemmatwo} with $k=13$, $\ell$ as given, $\idealuu=\id{a+6}$,
and $c=c_j+12d_{7,j}(a)$ gives $c_j+12d_{7,j}(a)\in\id{13,a+6}=\idealb$.
Thus, similarly to just above,
$f_{73} \equiv g_7(a) \bmod \idealb$ at the level of Fourier coefficients.
As in the proof of Theorem~\ref{cong61}, this proves that all
$p$-Euler polynomials of $f_{73}$ and~$g_7(a)$ are congruent
modulo~$\ideala$ and modulo~$\idealb$.

The norms of possible prime-power congruence ideals are limited to
$3$ or~$13$ by computing the norms from $\OK$ to~$\Z$ of two
eigenvalue differences and of their sum,
\begin{align*}
&\No\left( \evhoef\Ttwo{g_7(a)} - \evhoef\Ttwo{f_{73}}\right)   
=\No\left(a+6\right)= 2^2 \cdot 3 \cdot 5^2 \cdot 13 \cdot 2377, \\
&\No\left( \evhoef\Tthree{g_7(a)} - \evhoef\Tthree{f_{73}} \right) \\
&\quad = \No\left(  \tfrac{241}{4} - \tfrac{3171}{16} a + \tfrac{13965}{64} a^2 -
\tfrac{6297}{64} a^3 + \tfrac{167}{8} a^4 - \tfrac{133}{64} a^5
+\tfrac{5}{64} a^6 + 2 \right) \\
&\quad = 2^2 \cdot 3 \cdot 13 \cdot 195809, \\
&\No\left( \evhoef\Ttwo{g_7(a)} - \evhoef\Ttwo{f_{73}}
  +\evhoef\Tthree{g_7(a)} - \evhoef\Tthree{f_{73}} \right) \\
%=2770818921.
&\quad = 3^6 \cdot 13 \cdot 61 \cdot 4793,
\end{align*}
because these norms have greatest common divisor $3\cdot13$.
If the congruence ideal has norm~$3$ then because it contains the
eigenvalue difference $a+6$ it contains~$\id{3,a+6}=\id{3,a}=\ideala$,
so it is~$\ideala$.  Similarly, if the congruence ideal has norm~$13$
then it is~$\id{13,a+6}=\idealb$.
\end{proof}

The integers $d_{1,j}$ and the polynomials $d_{7,j}$
in Figure~\ref{table5} were determined similarly to the polynomials
$d_j$ in Figure~\ref{table4}.
The value $w$ in the proof was found by testing the set
$\{c_j+nd_{7,j}(a)\}$ for various integers~$n$ until one element of
the set divided all the others in~$\OK$; $w$ came from $n=2$ and then $j=3$.

We remark that the Gritsenko lift with $\Ttwo$-eigenvalue given
by~$q_1$ arises from the elliptic newform 73.4.a.a at the database
lmfdb, and similarly for~$q_7$ and 73.4.a.b.

%%   gamma*v  = gamma1/gamma2 * v =
%%   {103/2 - (257 L)/4 + (185 L^2)/8 - 3 L^3 + L^4/8, 
%%    199/4 - (435 L)/8 + (257 L^2)/16 - (7 L^3)/4 + L^4/16, -(19/4) + (
%%     87 L)/8 - (113 L^2)/16 + (5 L^3)/4 - L^4/16, 3 - L, 10 - 9 L + L^2, 
%%    4 - L, -(203/4) + (443 L)/8 - (257 L^2)/16 + (7 L^3)/4 - L^4/16}

\begin{theorem}\label{cong79}
Let $\GritSp \subseteq \CFsthreelev{79}$ be the subspace of Gritsenko lifts.  
The characteristic polynomial~$q$ of $\Ttwo$ on~$\GritSp$ factors
over~$\Q$ as $q=q_2q_5$, where its irreducible factors $q_2$ and~$q_5$
are
\begin{align*}
q_2(x)  &= x^2 -11x +26,  \\
q_5(x)  &= x^5 -27 x^4 +261 x^3 -1077 x^2 + 1766 x -964.  
\end{align*}

Let~$a$ be a root of~$q_2$ and~$K=\Q(a)$.  
With reference to the elements $d_{2,j}(a)$ of~$K$ in Figure~\ref{table6}
and to the Gritsenko lifts $G[j]$ from Proposition~\ref{prop79},
consider an element of~$\GritSp(K)$,
$$
g_2(a)=\sum_{j=1}^7 d_{2,j}(a) G[j].
$$
Then $g_2(a)$ lies in~$\GritSp(\OK)$,  
and it is a $\Ttwo$-eigenform with eigenvalue~$a$,
and it is an eigenform of~$\Tp$ and~$\Tozpsq$ for all primes~$p$.  
The $\OK$-ideal
$$
\ideala=\id{2, a+1}
$$
is prime, and the Fourier coefficients and the Euler polynomials of
$f_{79}$ and~$g_2(a)$ are congruent modulo~$\ideala$.
The only (proper) ideal that gives a congruence between the Euler 
polynomials of $f_{79}$ and~$g_2(a)$ is~$\ideala$. 

Let~$b$ be a root of~$q_5$ and~$L=\Q(b)$.  
With reference to the elements $d_{5,j}(b)$ of~$L$ in Figure~\ref{table6}
and to the Gritsenko lifts $G[j]$ from Proposition~\ref{prop79},
consider an element of~$\GritSp(L)$,
$$
g_5(b)= \sum_{j=1}^7 d_{5,j}(b) G[j].
$$
Then $g_5(b)$ lies in~$\GritSp(\OL)$,  
and it has a $\Ttwo$-eigenform with eigenvalue~$b$,
and it is an eigenform of~$\Tp$ and~$\Tozpsq$ for all primes~$p$.  
Let
$w=\tfrac{35}4 - \tfrac{435}8b + \tfrac{565}{16}b^2 - \tfrac{25}4b^3
+\tfrac5{16}b^4$, an element of~$\OL$.
The $\OL$-ideal
$$
\idealb=\id{8,w}
$$
is the cube of a prime $\OL$-ideal over~$2$ of norm~$4$.
The Fourier coefficients of $f_{79}$ and~$5g_5(b)$ are congruent
modulo~$\idealb$,
and the Euler polynomials of $f_{79}$ and~$g_5(b)$ are congruent
modulo~$\idealb$.
Every (proper) ideal that gives a congruence between the Euler 
polynomials of $f_{79}$ and~$g_5(b)$ divides~$\idealb$.
\end{theorem}

\begin{figure}[htb]
$$
\begin{array}{|c||c|c|c|}
\hline
j\,  & c_j        & d_{2,j}(a)           & d_{5,j\down\!\!\!}(b)     \\
\hline\hline
1\up\down          & 4        & 5-a
& \frac{103}{2} - \frac{257}{4}b + \frac{185}{8}b^2 - 3b^3 + 
  \frac{1}{8}b^4   \\
\hline
2\up\down        & 13   & -8+a
& \frac{199}{4} - \frac{435}{8}b + \frac{257}{16}b^2 - \frac{7}{4}b^3 + 
  \frac{1}{16}b^4      \\
\hline
3\up\down        & -15   & 6-a
& -\frac{19}{4} + 
  \frac{87}{8}b - \frac{113}{16}b^2 + \frac{5}{4}b^3 - \frac{1}{16}b^4
                 \\
\hline
4\up\down        & 8  & 3-a
& 3-b             \\
\hline
5\up\down        & 0   & 9-a
& 10  -9b +b^2 
                   \\
\hline
6\up\down       & 5  & 4-a  &  4-b
                \\
\hline
7\up\down       & -11 & -18+3a  &  -\frac{203}{4} + \frac{443}{8}b - 
  \frac{257}{16}b^2 + \frac{7}{4}b^3 - \frac{1}{16}b^4  
                   \\
\hline
\end{array}
$$
\caption{Coefficients of Gritsenko lifts for $N=79$\label{table6}}
\end{figure}

%Mathematica:
%Degree 5 irreducible factor disc: $2^{16}\cdot4787257$,
%factorization modulo~$2$: $x^2(x+1)^3$
%
%Using Sage to factor ideals \\
%Code: \\
%$A.<x> = QQ[]$ \\
%$f=-964 + 1766 x - 1077 x^2 + 261 x^3 - 27 x^4 + x^5$ \\
%$E. <b> = \operatorname{NumberField}(f)$ \\
%print(2) \\
%$I = E.\operatorname{fractional ideal}(2)$ \\
%print(I.factor()) \\
%$w= -(1/4) - (87 b)/8 + (113 b^2)/16 - (5 b^3)/4 + b^4/16$ \\
%print(w) \\
%$J = E.fractional_ideal(32,w)$ \\
%$print(J.factor())$ \\
%$print(J.norm())$ \\
%
%Output: \\
%$2$ \\
%$\phantom{* }(2, 1/32b^4 - 3/4b^3 + 189/32b^2 - 271/16b + 107/8)$ \\
%$* (2, 1/32b^4 - 3/4b^3 + 189/32b^2 - 271/16b + 115/8)$ \\
%$* (2, 1/16b^4 - 5/4b^3 + 129/16b^2 - 151/8b + 51/4)$ \\
%$w=1/16b^4 - 5/4b^3 + 113/16b^2 - 87/8b - 1/4$ \\
%$(32,w)=(2, 1/16b^4 - 5/4b^3 + 129/16b^2 - 151/8b + 51/4)$ \\
%$4$

\begin{proof}
Similarly to the proof of Theorem~\ref{cong61},
the matrix~$M$ of $\Ttwo$ on~$\GritSp$
for the basis $\{G[j]:j=1,\dotsc,7\}$ is
$$
M=
\left(
\begin{array}{rrrrrrr}
 0 & -2 & -2 & -2 & -8 & -2 & -6 \\
 -4 & 2 & 0 & -3 & -7 & -2 & -5 \\
 3 & 0 & 9 & 1 & 2 & 2 & 4 \\
 0 & -1 & 0 & 11 & -1 & -6 & -1 \\
 3 & 0 & 1 & 2 & 9 & 2 & 5 \\
 0 & -1 & 0 & 7 & -1 & -3 & -1 \\
 5 & 2 & 1 & -3 & 8 & 6 & 10 \\
\end{array}
\right)
$$
The characteristic polynomial~$q(x)=\det(xI-M)=q_2(x)q_5(x)$ is as
stated.

Let~$a$ be a root of~$q_2$ and $K=\Q(a)$.    
Each $d_{2,j}(a)$ lies in~$\OK$
and each $G[j]$ lies in $\GritSp(\Z)$,
so $g_2(a)$ lies in~$\GritSp(\OK)$.
The vector $(d_{2,1}(a), \dotsc, d_{2,7}(a))$
lies in $\operatorname{null}(aI-M)$, and so $g_2(a)$ is a
$\Ttwo$-eigenform in~$\GritSp(\OK)$ with eigenvalue~$a$.
The characteristic polynomial~$q$ of~$\Ttwo$ is separable because its
discriminant
$\disc(q)=2^{26}\cdot17^3\cdot59^2\cdot4787257$ is nonzero,
and so $g_2(a)$ is an eigenform of all the stated Hecke operators
as in the proof of Theorem~\ref{cong61}.

The reduction of~$q_2(x)$ modulo~$2$ is $q_2(x)\equiv x(x+1)\bmod2$,
and because $2$ does not divide $\disc(q_2)=17$,
Lemma~\ref{mirrorfactorlemma} says that
the $\OK$-ideal $\ideala=\id{2,a+1}$ is prime.

The values~$c_j$ in Figure~\ref{table6} are such that
$f_{79} = -32 {\idealf} + \sum_{j=1}^7 c_j G[j]$ in Proposition~\ref{prop79}.
Thus
$$
f_{79}-g_2(a) = -32 {\idealf} + \sum_{j=1}^7 (c_j-d_{2,j}(a)) G[j].
$$
For $j=1, \dotsc, 7$ we compute that $53(c_j-d_{2,j}(a)) \in \id{a+5}$.  
Lemma~\ref{lemmatwo} with $k=2$, $\ell=53$, $\idealuu=\id{a+5}$,
and $c=c_j-d_{2,j}(a)$ gives $c_j-d_{2,j}(a)\in\id{2,a+5}=\ideala$.   
The Fourier coefficients of~$-32\idealf$ lie in~$\ideala$ as well.
Thus $f_{79} \equiv g_2(a) \bmod \ideala$ at the level of Fourier coefficients.  
As in the proof of Theorem~\ref{cong61}, this proves that all
$p$-Euler polynomials of $f_{79}$ and~$g_2(a)$ are congruent modulo~$\ideala$.

The possible congruence ideals for $f_{79}$ and~$g_2(a)$ are limited
to norm-$2$ ideals by computing the norm of an eigenvalue differences
and then a second eigenvalue difference that is already a rational integer,
\begin{align*}
\No\left( \evhoef\Ttwo{g_2(a)}-\evhoef\Ttwo{f_{79}} \right)
&=\No(a+5)=2\cdot 53,   \\
\evhoef\Tthree{g_2(a)}-\evhoef\Tthree{f_{79}}
&=11-(-5)=16=2^4,
\end{align*}  
because the greatest common divisor of these values is~$2$.
Any such congruence ideal also contains the eigenvalue difference
$a+5$, so it contains~$\id{2,a+5}=\ideala$, and so it is~$\ideala$.

Let~$b$ be a root of~$q_5$ and $L=\Q(b)$.
Each $d_{5,j}(b)$ lies in~$\OL$
and each $G[j]$ lies in $\GritSp(\Z)$,
so $g_5(a)$ lies in~$\GritSp(\OL)$.
The vector $(d_{5,1}(b),\dotsc,d_{5,7}(b))$
lies in $\operatorname{null}(bI-M)$., and so $g_5(b)$ is a
$\Ttwo$-eigenform in $\GritSp(\OL)$ with eigenvalue~$b$.
Because the roots of~$q_5$ are distinct,
$g_5(b)$ is an eigenform of all the stated Hecke operators.

%% [more data requested]
%% lambda_2(Grit) = L = root of -964 + 1766 x - 1077 x^2 + 261 x^3 - 27 x^4 + x^5
%% lambda_3(Grit) = 1/16 (-764 + 1014 L - 273 L^2 + 28 L^3 - L^4)
%% lambda_5(Grit) = 1/16 (1884 - 1582 L + 401 L^2 - 36 L^3 + L^4)
%% lambda_7(Grit) = 1/16 (2396 - 1254 L + 289 L^2 - 28 L^3 + L^4)
%% lambda_11Grit) = 1/8 (612 + 110 L + 57 L^2 - 16 L^3 + L^4)
%% lambda_13(Grit) = 1/16 (8756 - 7106 L + 2479 L^2 - 332 L^3 + 15 L^4)
%% 
%% lambda_2(degree 5 Grit) - lambda_2(f) has norm that factors into
%% {{2, 8}, {349, 1}}
%% lambda_3(degree 5 Grit) - lambda_3(f) has norm that factors into
%% {{2, 10}, {617, 1}}
%% lambda_5(degree 5 Grit) - lambda_5(f) has norm that factors into
%% {{2, 6}, {5, 1}, {7, 2}, {67, 1}}
%% These norms have GCD 2^6  

Similarly to just above,
$$
f_{79}-5g_5(b) = -32 {\idealf} + \sum_{j=1}^7 (c_j-5d_{5,j}(b)) G[j].
$$
Recall the element~$w$ of~$\OL$ and the ideal $\idealb=\id{8,w}$ from
the statement of the theorem.
Computer software says that $\idealb$ is the cube of the prime,
norm~$4$ ideal $\id{2,v}$ over~$2$, where
$v=\tfrac1{16}b^4-\tfrac54b^3+\tfrac{129}{16}b^2-\tfrac{151}8b+\tfrac{51}4$,
%a prime ideal over~$2$ of norm~$4$,
and so $\idealb$ has norm~$64$ as also can be confirmed directly.
For $j=1, \dotsc, 7$, we compute $635(c_j-5d_{5,j}(b)) \in\id w$.  
Lemma~\ref{lemmatwo} with $k=8$, $\ell=635$, $\idealuu=\id w$,
and $c=c_j-5d_{5,j}(b)$ give $c_j-5d_{5,j}(b) \in \id{8,w}=\idealb$.  
The Fourier coefficients of~$-32\idealf$ lie in~$\idealb$ as well.
Thus $f_{79} \equiv 5g_5(b) \bmod \idealb$ at the level of Fourier
coefficients.
As in the proof of Theorem~\ref{cong61}, and noting that scaling
$g_5(b)$ by~$5$ has no effect on its Hecke eigenvalues or Euler
polynomials, this proves that all $p$-Euler polynomials of $f_{79}$
and~$g_5(a)$ are congruent modulo~$\idealb$.

The possible congruence ideals for $f_{79}$ and~$g_5(a)$ are limited
to divisors of~$\idealb$ by computing the norms of two eigenvalue
differences,
\begin{align*}
&\No\left( \evhoef\Ttwo{g_5(b)} - \evhoef\Ttwo{f_{79}} \right)   
=\No\left( b+5 \right)= 2^8 \cdot 349,  \\
&\No\left( \evhoef\Tfive{g_5(b)} - \evhoef\Tfive{f_{79}} \right) \\
&\quad=\No\left( \tfrac{1884}{16} - \tfrac{1582}{16} b + \tfrac{401}{16} b^2 -
\tfrac{36}{16} b^3 + \tfrac{1}{16} b^4   - 3 \right) \\
&\quad=2^6 \cdot 5 \cdot 7^2 \cdot 67, 
\end{align*}
because the greatest common divisor of these norms is~$64$,
and so the congruence ideal has norm dividing~$64$.
Letting $\idealc$ denote the congruence ideal,
also the least common multiple of $\idealb$ and~$\idealc$ is a
congruence ideal, so its norm divides~$64=\No(\idealb)$,
so the least common multiple is~$\idealb$.
That is, $\idealc\mid\idealb$.
\end{proof}

The $w$ in this proof was found similarly to the $N=73$ case,
this time with $n=5$ and $j=3$.
The Gritsenko lifts with $\Ttwo$-eigenvalues given by~$q_2$ and~$q_5$
arise from the elliptic newforms 79.4.a.a and 79.4.a.b at the lmfdb.

%%%%%%%%%%%%%%%%%%%%%%%%%%%%%%%%%% 
%%%%%%%%%%%%%%%%%%%%%%%%%%%%%%%%%% 
%%%%%%%%%%%%%%%%%%%%%%%%%%%%%%%%%% 
%%%%%%%%%%%%%%%%%%%%%%%%%%%%%%%%%%
\section{Computation of eigenvalues\label{sectioneigenvalues}}
%%%%%%%%%%%%%%%%%%%%%%%%%%%%%%%%%%
%%%%%%%%%%%%%%%%%%%%%%%%%%%%%%%%%%
%%%%%%%%%%%%%%%%%%%%%%%%%%%%%%%%%%
%%%%%%%%%%%%%%%%%%%%%%%%%%%%%%%%%%

\newcommand\smat{\begin{psmallmatrix}a&b\\b&c/N\end{psmallmatrix}}
\newcommand\ssp{\widecheck s}
\newcommand\sSp{\widehat s}
  
Throughout this section, $N$ is a positive integer, $f$ is an
element of~$\MFskN$, and $p$ is prime.
Further, $a,b,c$ are integers such that
a matrix and two of its $\SLtwo{\Q(\sqrt p)}$-equivalents are
positive,
$$
s=\begin{pmatrix}a&b\\b&c/N\end{pmatrix},
\qquad
\ssp=\begin{pmatrix}a/p&b\\b&pc/N\end{pmatrix},
\qquad
\sSp=\begin{pmatrix}pa&b\\b&c/(pN)\end{pmatrix}.
$$
Let $\phi_s(\tau)=\tau s$ for~$\tau\in\UHP$,
and let $\phi_s^*$ denote its pullback.
Proposition~\ref{restrictionbigsums} to follow gives initial formulas
for the restrictions $\phi_s^*(f\wtk\Tp)$ and $\phi_s^*(f\wtk\Topsq)$
when $p\nmid N$,
and similarly Proposition~\ref{prop:badprime-restrict} for
$\phi_s^*(f\wtk\Tp)$ and $\phi_s^*(f\wtk\Tozpsq)$ when $p\mmid N$.
Each of these initial formulas consists of finitely many finite sums,
but some of the sums are computationally intractable.
Propositions~\ref{speedup1} and~\ref{speedup-badprime} show that
various sums in the initial formulas can be replaced by sums
over smaller index sets, making them tractable after all.
These ideas were used in~\cite{bpptvy18}, but the details were not
given there.

\begin{lemma}
With reference to the matrix~$s$ just above, define a map from the
complex upper half plane to the $2$-dimensional Siegel upper half
space,
$$
\phi_s:\UHP\lra\UHP_2,\qquad \phi_s(\tau)=\tau s.
$$
Let $R\subseteq\C$ be a subring.  
Then the pullback of $\phi_s$ is a ring homomorphism from the graded
ring of Siegel paramodular forms of level~$N$ with coefficients in~$R$
to the graded ring of elliptic modular forms of level $\det(s)N$ with
coefficients in~$R$,
\begin{equation*}
\phi_s^*:\MFsnoweight(\KN)(R)\lra\MFsnoweight(\Gamma_0(\det(s)N))(R)
\end{equation*}
given by
$$
(\phi_s^*f)(\tau)=f(\tau s).
$$
The map $\phi_s^*$ multiplies weights by~$2$ and takes cusp forms to
cusp forms.
\end{lemma}

The elliptic modular form~$\phi_s^*f$ is the {\bf restriction} of~$f$
to the curve $\phi_s(\UHP)$, also called the restriction of~$f$
under~$s$.

\begin{proof}
The proof follows from a straightforward modification of a result of
Poor--Yuen \cite[Proposition 5.4]{MR2379329}.
% Theorem LLLL in [PY02] or Theorem MMMM in [PY07].  
%%\bibitem{PY07}
%%Cris Poor and David S.\ Yuen, \emph{Computations of spaces of Siegel
%%modular cusp forms}, J.\ Math.\ Soc.\ Japan %%\textbf{59} (2007),
%%no.\ 1, 185--222.
\end{proof}

\newcommand\blockmat{\begin{psmallmatrix}A&B\\0&D\end{psmallmatrix}}

Let the paramodular form~$f\in\MFskN$ have Fourier expansion (in which
$\ip t\Omega=\tr(t\Omega)$)
$$
f(\Omega)=\sum_{t\in\XtwoNsemi}\fc tf\,\e(\ip t\Omega).
$$
Its restriction $\phi_s^*f\in\MFswtgp{2k}{\Gamma_0(\det(s)N)}$ has
Fourier expansion (in which $q=\e(\tau)$)
$$
(\phi_s^*f)(\tau)=\sum_{n=0}^{\infty}
\biggl(\sum_{\substack{\,t: \,\ip st=n}}\fc tf\biggr) q^n.
$$
Furthermore, if $f$ is slashed with a block upper triangular matrix
$\blockmat\in\GSp_4^+(\Q)$ with similitude $\mu = \det(AD)^{1/2}$ then
the restriction of the resulting function is
\begin{equation} \label{eqn:phisslash}
\begin{aligned}
&\phi_s^*(f\wtk{\blockmat})(\tau)
=(f\wtk{\blockmat})(s\tau) \\
&\qquad=\det(AD)^{k-3/2}\det(D)^{-k}f(AsD^{-1}\tau+BD^{-1}) \\
&\qquad=\det(A)^{k}\det(AD)^{-3/2}\sum_{n \in \Q_{\geq 0}}
\biggl(\sum_{\substack{t:\,\ip{AsD^{-1}}t=n}}
\fc tf\e\left(\ip{BD^{-1}}t\right)\biggr) q^n.
\end{aligned}
\end{equation}

%%  slide
In order to compute eigenvalues by the technique of restriction to a
modular curve,
we apply a restriction map~$\phi_s^*$ to the eigenvalue equation 
$\lambda_f(T) f = f \vert T = \sum_{j=1}^m f \vert t_j$.  
Here we assume $T=\KN\diag(a,b,c,d)\KN= \bigsqcup_{j=1}^m \KN t_j$ 
where the $t_j$ are upper block triangular and $m= \deg T$ is the 
number of cosets in $\KN\backslash T$.  Using equation~\eqref{eqn:phisslash}
for each term~$\phi_s^*(f \vert t_j)$, the restricted eigenvalue equation 
\begin{equation} \label{eqn:phisslashtwo}
\lambda_f(T) \phi_s^*(f)  = \sum_{j=1}^m \phi_s^*(f \vert t_j )  
\end{equation}
uniquely determines the eigenvalue $\lambda_f(T)$ as long as the elliptic 
modular form $\phi_s^*(f)$ does not vanish identically.  
Indeed, each successive power of~$q=\e(\tau)$ in equation~\eqref{eqn:phisslashtwo}
provides an independent evaluation of the eigenvalue~$\lambda_f(T)$ 
and is thus useful for checking computational infrastructure.  
For efficiency, the coefficients of equation~\eqref{eqn:phisslashtwo}
are evaluated over a finite field~$\F_{\ell}$ rather than over~$\Qbar$.  

Let $T$ have a similitude~$\mu$ that is a $p$-power.  
The factors $\e\left(\ip{BD^{-1}}t\right)$ in equation~\eqref{eqn:phisslash}
are $\mu$-th roots of unity.  
For simplicity assume $\lambda_f(T) \in \Z$.  
Choose an auxiliary prime~$\ell$ that 
splits completely in the cyclotomic field $K=\Q(\e(1/\mu))$.  
By Kummer--Dedekind, Lemma~\ref{mirrorfactorlemma} here, 
the $\mu$-th cyclotomic 
polynomial $\Phi_{\mu}$ splits in~$\F_{\ell}$.  
Let $r \in \Z$ give a root of $\Phi_{\mu}$ in~$\F_{\ell}$.   
In $K$ we know that 
$\No\left(  r - \e(1/\mu) \right) = \Phi_{\mu}(r) \equiv 0 \bmod \ell$, 
so that there is a prime ideal ${\frak m}$ in $\OK$  above $\langle  r - \e(1/\mu), \ell \rangle$ 
by Lemma~\ref{lemmaone}.  We evaluate the coefficients of equation~\eqref{eqn:phisslashtwo}
over the finite field $\OK/{\frak m} \cong \F_{\ell}$, using the congruence 
$\e(1/\mu) \equiv r \bmod {\frak m}$ to reduce the computation to integers, 
and obtain $\lambda_f(T) \bmod \ell$.  For sufficiently large~$\ell$, 
the general bound $\vert \lambda_f(T) \vert \le \mu^{k-3} \deg T$, 
compare Proposition~{6.7.1} in~\cite{bpptvy18}, determines  $\lambda_f(T) \in \Z$.  

The main computational advantage of restricting to modular curves, 
as opposed to using the formulae of Lemma~\ref{lemmafour}, 
is that many terms in equation~\eqref{eqn:phisslashtwo} may be omitted if we project 
onto integral powers of~$q$ after a partial summation.  
The speed-ups in Propositions~\ref{speedup1} and~\ref{speedup-badprime}    
prove that we may partially sum over index sets that are roughly a factor of~$p$ 
smaller than $\deg T$ and still preserve equality for integral powers of~$q$ in equation~\eqref{eqn:phisslashtwo}.  

A secondary benefit is that we may not need to compute Fourier coefficients of the eigenform~$f$.  For example, 
in section~\ref{sectionconstruction} each eigenform was given as a rational function of Gritsenko lifts~$G[i]$.  
The specializations $\phi_s^*(G[i] \vert t_j)$ are computed from the Fourier coefficients of these 
Gritsenko lifts, which may be reduced to computing Fourier coefficients of Jacobi forms.  
It is also computationally beneficial to use single rather than multiple variable power series.  
The~$s$ actually used to restrict~$f$ can be selected from a number of candidates for speed and 
to make the $q$-order of $\phi_s^*(f)$ small.  For $N=61$ and good primes~$p$, we used 
$s=\smallmat{122}{11}{11}{1}$, which gave $q$-order~$2$; for $p=61$, we used 
$s=\smallmat{13}{2}{2}{19/61}$ with $q$-order~$1$.  
For $N=73, 79$, in all cases 
we used $s=\smallmat{146}{17}{17}{2}$, $\smallmat{158}{47}{47}{14}$, respectively, 
each with $q$-order~$3$.
\smallskip

We now write speed-up theorems for computing the restrictions of
$f\wtk\Tp$ and $f\wtk\Topsq$ for $p\nmid N$ and of
$f\wtk\Tp$ and $f\wtk\Tozpsq$ for $p\mmid N$.
Recall that these Hecke operators are defined as slashes by double
cosets,
\begin{align*}
\Tp&=\KN\diag(1,1,p,p)\KN \\
\Topsq&=\KN\diag(1,p,p^2,p)\KN \\
\Tozpsq&=\KN \diag(p,p^2,p,1)\KN.
\end{align*}
See \cite{bpptvy18,schmidt18} for the decompositions of these double
cosets into right cosets.

For the case $p\nmid N$, we use the single coset decomposition
from~\cite{bpptvy18} and the following result after
applying~\eqref{eqn:phisslash}.

\begin{proposition} \label{restrictionbigsums}
Let $N,f,p,s,\ssp,\sSp$ be as at the beginning of this section.
Let $p\nmid N$.
For any integers $i,j,k$ let
$t_{i,j,k}=\begin{psmallmatrix}i/p&j/p\\j/p&k/p\end{psmallmatrix}$,
$u_{i,j,k}=\begin{psmallmatrix}i/p^2&j/p\\j/p&k/p^2\end{psmallmatrix}$,
and $v_i=pu_{0,ia,i(ia+2b)}$.
The restrictions of $f\wtk{\Tp}$ and $f\wtk{\Topsq}$ under~$s$ are
\begin{equation*} %\label{eqn:restriction1}
\begin{aligned}
&\phi_s^*(f\wtk\Tp)(\tau) = p^{2k-3} f(ps\tau)
% \text{\quad (Type I term)}
+ p^{k-3}\sum_{i \bmod{p}}
f(\ssp\tau+t_{i,0,0}) \\
% \text{\qquad (Type II terms)}
&\qquad + p^{k-3}\sum_{i,k\bmod p}f((\sSp+v_i)\tau+t_{0,0,k})
%\text{\quad (Type III terms)}
+ p^{-3}\sum_{i,j,k \bmod{p}}f(s\tau/p+t_{i,j,k})
% \text{\qquad (Type IV terms)}
\end{aligned}
\end{equation*}
and
\begin{equation*} %\label{eqn:restrictionT1p2}
\begin{aligned}
&\phi_s^*(f\wtk\Topsq)(\tau) 
= p^{3k-6} f(p\ssp\tau)
% \text{\quad (Type I term)}
+ p^{3k-6}\sum_{i \bmod{p}}
f(p(\sSp+v_i)\tau) \\
 %\text{\qquad (Type II terms)}
&\qquad + p^{2k-6}\sum_{i\not\equiv 0 \bmod{p}}
f(s\tau+t_{i,0,0})
 %\text{\qquad (Type III terms)}
+ p^{2k-6}\sum_{\substack{i \bmod{p}, \\ j\not\equiv 0 \bmod{p}}}
f(s\tau+jt_{i^2,i,1}) \\
% \text{\qquad (Type IV terms)}
&\qquad + p^{k-6}
\sum_{\substack{i\bmod{p^2}, \\ j \bmod{p}}}
f(\ssp\tau/p+u_{i,j,0})
 %\text{\quad (Type V terms)}
+ p^{k-6}\sum_{\substack{i,j \bmod{p}, \\ k\bmod{p^2}}} 
f((\sSp+v_i)\tau/p+u_{0,j,k})
% \text{\qquad (Type VI terms)}
\end{aligned}
\end{equation*}
\end{proposition}

Upon expanding in Puiseux $q$-series, there is cancellation within
the sums of restrictions in Proposition~\ref{restrictionbigsums}.
The following proposition, which repeats Proposition~6.3.8
of~\cite{bpptvy18} but also includes some additional formulas,
shows that partial summation gives new restrictions whose sum
{\em over smaller index sets\/} equals the original sum for integral
powers of~$q$.
The proposition is subtle in that its simpler coefficients
necessarily match the original ones only at integral powers.
For a Puiseux series $f\in\C[[q^{1/\infty}]]$ and $e\in\Q_{\geq 0}$,
let $\Coeff_ef$ denote the coefficient of $q^e$ in~$f$, a complex number.

\begin{proposition}\label{speedup1}
Let $N,f,p,s,\ssp,\sSp$ be as at the beginning of this section.
Let $p\nmid N$.
For any integers $i,j,k$ let
$t_{i,j,k}=\begin{psmallmatrix}i/p&j/p\\j/p&k/p\end{psmallmatrix}$,
$u_{i,j,k}=\begin{psmallmatrix}i/p^2&j/p\\j/p&k/p^2\end{psmallmatrix}$,
and $v_i=pu_{0,ia,i(ia+2b)}$.
The following statements hold for all $e\in\Znn$.
\begin{enumalph}
\item If $p \nmid a$ then
\begin{align*}% \label{eqn:restriction2}
\Coeff_e \sum_{i \bmod{p}}
f(\ssp\tau+t_{i,0,0}) &= p \Coeff_e f(\ssp\tau)
\intertext{and}
\Coeff_e \sum_{i,j,k \bmod{p}}f(s\tau/p+t_{i,j,k})
&=p\Coeff_e\sum_{j,k \bmod{p}}f(s\tau/p+t_{0,j,k})
\intertext{and}
\Coeff_e \sum_{\substack{i\bmod{p^2}\\ j \bmod{p}}} f(\ssp\tau/p+u_{i,j,0})
&=p^2 \Coeff_e \sum_{\substack{ j \bmod{p}}} f(\ssp\tau/p+u_{0,j,0}).
\end{align*}
\item If $p \nmid b$ then
\begin{align*}
\Coeff_e \sum_{i,j,k \bmod{p}}f(s\tau/p+t_{i,j,k})
&=p\Coeff_e \sum_{i,k \bmod{p}}f(s\tau/p+t_{i,0,k})\\
\intertext{and}
\Coeff_e\sum_{\substack{i\bmod{p^2}\\ j \bmod{p}}} f(\ssp\tau/p+u_{i,j,0})
&=p\Coeff_e\sum_{\substack{ i \bmod{p^2}}}f(\ssp\tau/p+u_{i,0,0}).
\end{align*}
\item If $p \nmid c$ then
\begin{equation*}
\Coeff_e \sum_{i,j,k \bmod{p}}f(s\tau/p+t_{i,j,k})
=p\Coeff_e \sum_{i,j \bmod{p}}f(s\tau/p+t_{i,j,0}).
\end{equation*}
\item
For $i \in \Z$, if $p  \nmid  c+i(ia+2b)N$ then
\begin{align*}
\Coeff_e \sum_{k \bmod p}f((\sSp+v_i)\tau+t_{0,0,k})
&=p\Coeff_ef((\sSp+v_i)\tau) \\
\intertext{and}
\Coeff_e\sum_{\substack{j\bmod{p}\\k \bmod{p^2}}}f((\sSp+v_i)\tau/p+u_{0,j,k})
&=p^2\Coeff_e\sum_{\substack{j \bmod{p}}}f((\sSp+v_i)\tau/p+u_{0,j,0}).
\end{align*}
\end{enumalph}
\end{proposition}

\begin{proof}
We prove the second part of~(d), the others being similar;
(c) is proved in~\cite{bpptvy18}.
Let $p\nmid c+i(ia+2b)N$.  Let $e\in\Znn$.
With $\myT=\begin{psmallmatrix}n&r/2\\r/2&mN\end{psmallmatrix}$,
the coefficient of $q^e$ on the left side is
\begin{equation*} %\label{eqn:ijkpsum}
\sum_{\substack{j\bmod{p},\ k\bmod{p^2}\\
n,r,m:\, an+(ia+b)r/p+(c+i(ia+2b)N)mN/p^2=e}}
\fc\myT f \e(jr/p+kmN/p^2).
\end{equation*}
Because $\sum_{j\bmod{p}}\e(jr/p)=0$ if $p\nmid r$, this sum is
\begin{equation*} %\label{eqn:ijkpsum2}
\begin{aligned}
\sum_{\substack{
n,r,m:\, an+(ia+b)r/p\\
\qquad\qquad+(c+i(ia+2b)N)mN/p^2=e
\\
p\mid r}
}
\ \sum_{\substack{j\bmod{p},\ k\bmod{p^2}}}
\fc\myT f \e(jr/p+kmN/p^2).
\end{aligned}
\end{equation*}
Because $p \nmid  c+i(ia+2b)N$, if $p|r$ then $p^2\mid m$ inside the
summation (since $p\nmid N$).  Thus the above sum becomes
\begin{align*}
&\sum_{\substack{
n,r,m:\, an+(ia+b)r/p\\
\qquad\qquad+(c+i(ia+2b)N)mN/p^2=e
\\
p|r}
}
\sum_{\substack{j\bmod{p},\ k\bmod{p^2}}}
\fc\myT f\e(jr/p+ 0 ) \\
&\qquad\qquad =
\sum_{\substack{
n,r,m:\, an+(ia+b)r/p\\
\qquad\qquad+(c+i(ia+2b)N)mN/p^2=e
\\
p|r}
}
\sum_{\substack{j\bmod{p}}}
p^2\fc\myT f\e(jr/p )
 \\   
&\qquad\qquad =
\sum_{\substack{
n,r,m:\, an+(ia+b)r/p\\
\qquad\qquad+(c+i(ia+2b)N)mN/p^2=e
}
}
\sum_{\substack{j\bmod{p}}}
p^2\fc\myT f\e(jr/p )
 \\   
&\qquad\qquad =
p^2\sum_{j\bmod{p}}\sum_{\substack{
n,r,m:\, an+(ia+b)r/p\\
\qquad\qquad+(c+i(ia+2b)N)mN/p^2=e
}
}
\fc\myT f\e(jr/p)
 \\   
&\qquad\qquad=p^2\sum_{j \bmod{p}}\Coeff_ef((\sSp+v_i)\tau/p+u_{0,j,0}).
\end{align*}
\end{proof}

\def\aa{\hat p}
\def\cc{\hat N}
We now give similar speed-up theorems for the case when $p\mmid N$.
Having only one power of~$p$ divide~$N$ is needed to have all upper
triangular coset representatives.

\begin{proposition}\label{prop:badprime-cosets}
Let $p\mmid N$.
Fix $\aa,\cc\in\Z$ such that $\aa p + \cc N/p = 1$.
We have the following right coset decompositions.
\begin{equation*} %\label{eqn:heckeopersbad1}
\begin{aligned}
&\KN \diag(p,p,1,1) \KN
= \sum_{\substack{i,j,k\bmod{p}}}
\KN 
\begin{psmallmatrix} 
1 &  0  &  i &  j \\
0  &  1  &  j  &  k/p  \\ 
    &         &  p  &  0  \\
    &         &  0  &  p 
\end{psmallmatrix} \\
&\qquad + \sum_{i,j \bmod{p}}
 \KN 
\begin{psmallmatrix} 
p  &  0  &  0  &  0  \\
i  &  1  &  0  &  j/p  \\ 
    &         &  1  &  -i  \\
    &         &  0  &  p 
\end{psmallmatrix}
+ \sum_{i,j \bmod{p}} \KN 
\begin{psmallmatrix} 
1  &  - N j  & i &  0  \\
0  &  p   &  0  & 0\\ 
    &         &  p & 0  \\
    &         &   N j   &  1 
\end{psmallmatrix} \\
&\qquad 
+  \sum_{\substack{j\bmod{p}   }}  
 \KN 
\begin{psmallmatrix} 
p  &  -N  &  jN/p &  j  \\
\cc  &  \aa p   &  \aa j &  -\cc j/p  \\ 
    &         &  \aa p  &  -\cc  \\
    &         &  N  &  p 
\end{psmallmatrix}
+ \KN 
\begin{psmallmatrix} 
p  &  0  &     &     \\
0 &  p &     &     \\ 
    &         &  1  &  0 \\
    &         &  0  &  1 
\end{psmallmatrix},
\end{aligned}
\end{equation*}
\begin{equation*} %\label{eqn:heckeopersbadTp2}
\begin{aligned}
&\KN \diag(p,p^2,p,1) \KN
= \sum_{\substack{i,j\bmod{p}\\k\bmod{p^2}}}
\KN 
\begin{psmallmatrix} 
p&  0  &  0&  jp \\
i  &  1  &  j  &  -ij+k/p  \\ 
    &         &  p  &  -ip  \\
    &         &  0  &  p^2 
\end{psmallmatrix} \\
&\qquad + \sum_{j \bmod{p}}
 \KN 
\begin{psmallmatrix} 
p  &  {- N j p}  &   &    \\
0  &  p^2   &    & \\ 
    &         &  p & 0  \\
    &         &   N j   &  1 
\end{psmallmatrix}
+  \sum_{\substack{i,j,k\bmod{p}\\k\ne0\bmod{p}}}
 \KN 
\begin{psmallmatrix} 
p  &  -j Np        &  0                 &  0 \\
-i  &  i j  N + p  &  j N k/p    &  k/p \\ 
    &                    &  i j N + p  &  i \\
    &                    &  j N p        &  p 
\end{psmallmatrix} \\
&\qquad + \sum_{\substack{i,j \bmod{p}\\   i\ne0  \bmod{p}}} \KN 
\begin{psmallmatrix} 
p     &  -Npj       &  -ij N                    &  -i  \\
\cc  &  p-\cc Nj  &  i(\cc N j/p-1)   &  i\cc /p  \\ 
      &               &  p-\cc Nj                  &  -\cc  \\
      &               &  j N p                   &  p 
\end{psmallmatrix}.
\end{aligned}
\end{equation*}
\end{proposition}

\def\ihat{{\hat\imath}}
\begin{proof}
The coset representatives of the decomposition of $\KN \diag(p,p,1,1)
\KN$ are precisely the right coset representatives given in
Proposition 2.10 of \cite{schmidt18}, except that we have replaced the
last representative in Proposition~2.10 of~\cite{schmidt18} as follows:
for~$p\nmid i$, replace
\begin{equation*}
\left(\begin{matrix} 
p &  0  &     &     \\
0 &  p  &     &     \\ 
  &     &  1  &  0 \\
  &     &  0  &  1 
\end{matrix}\right)
\left(\begin{matrix} 
1       &  0       &     &     \\
0       &  1       &     &     \\ 
0       &  i\cc N  &  1  &  0  \\
i\cc N  &  0       &  0  &  1 
\end{matrix}\right)
\quad\text{with}\quad
\left(\begin{matrix} 
p  &  -N  &  \ihat N/p &  \ihat \\
\cc  &  \aa p   &  \aa \ihat &  -\cc \ihat/p  \\ 
    &         &  \aa p  &  -\cc  \\
    &         &  N  &  p 
\end{matrix}\right),
\end{equation*}
where $\ihat$ is such that $i\ihat\equiv1\bmod{p}$.
It is a straightforward calculation that the left-hand representative
multiplied on the right by the inverse of the right-hand
representative is
\begin{equation*}
\left(\begin{matrix}
 (i \ihat \cc N+p)/p
  & (i \ihat \cc N^2- Np)/p^2
  & -\ihat  N/p
  & -\ihat \\
  (\cc p-i \ihat \cc^2N)/p^2
  & ((i \ihat \cc N +p)\aa)/p
  & -\ihat \aa
  & \ihat \cc/p \\
  (i \cc^2  N)/p
  & -i \cc  N \aa
  & p \aa
  & -\cc \\
  -i \cc  N
  & -(i \cc N ^2)/p
  &  N
  & p \\
\end{matrix}\right).
\end{equation*}
Using  the fact that $1 - i\ihat \cc N/p$ is a multiple of $p$, it is
straightforward to show that the entries satisfy the conditions for
the matrix to be in $\KN$.
Summing over $i\bmod{p}$, $i\ne0$ is the same as summing over
$\ihat\bmod{p}$, $\ihat\ne0$, and so we replace $\ihat$ with $j$.
Thus we may replace the representative as stated.
The proof of the decomposition of $\KN \diag(p,p^2,p,1) \KN$ is similar.
\end{proof}

In the next two propositions, $t_{i,j,k}$ and~$u_{i,j,k}$ are defined
differently than they were in Propositions~\ref{restrictionbigsums}
and~\ref{speedup1}.

\begin{proposition}\label{prop:badprime-restrict}
Let $N,f,p,s,\ssp,\sSp$ be as at the beginning of this section.
Let $p\mmid N$.
For any integers $i,j,k$ let
$t_{i,j,k}=\begin{psmallmatrix}i/p&j/p\\j/p&k/p^2\end{psmallmatrix}$,
$u_{i,j,k}=\begin{psmallmatrix}i/p&j/p\\j/p&k/p^3\end{psmallmatrix}$,
$v_i=pt_{0,ia,i(ia+2b)}$,
and $w_j=j\begin{psmallmatrix}(-2b + jc)N/p&-c\\-c&0\end{psmallmatrix}$.
Fix $\aa,\cc\in\Z$ such that $\aa p + \cc N/p = 1$.
The restrictions of $f\wtk\Tp$  and $f\wtk\Tozpsq$ are
\begin{equation*} %\label{eqn:T01restrict}
\begin{aligned}
&\phi_s^*(f\wtk\Tp)(\tau) \\
&= p^{-3} \sum_{i,j,k \bmod{p}}
f(s\tau/p + t_{i,j,k}) \\
% \text{\quad (Type I term)}
&\quad + p^{k-3}\sum_{i,k \bmod{p}}
f((\sSp+v_i)\tau+t_{0,0,k}) \\
 %\text{\qquad (Type II terms)}
&\quad + p^{k-3}\sum_{i,j \bmod{p}}
f((\ssp+w_j)\tau+t_{i,0,0}) \\
 %\text{\qquad (Type III terms)}
&\quad + p^{k-3}\sum_{\substack{j\not\equiv 0 \bmod{p}}}
f\Bigl(
\begin{psmallmatrix} 
-2 b N + c N/p + a p &  b - \aa c + a \cc -2 b \cc N/p   \\
b - \aa c + a \cc - 2 b \cc N/p  &  
\frac{-\aa c \cc N + a \cc^2 N + \aa c p + 2 \aa b \cc Np}{Np}  
\end{psmallmatrix} 
\tau+t_{0,j,0}\Bigr) \\
% \text{\qquad (Type IV terms)}
& \quad + p^{2k-3}f(ps\tau)
\end{aligned}
\end{equation*}
\begin{equation*} %\label{eqn:T10restrict}
\begin{aligned}
&\phi_s^*(f\wtk\Tozpsq)(\tau) \\
&= p^{k-6} 
\sum_{\substack{i,j \bmod{p}, \\ k \bmod{p^2}}}
f((\sSp+v_i)\tau/p+u_{0,j,k}) \\
% \text{\quad (Type I term)}
&\quad + p^{2k-6}\sum_{i,j\bmod{p}, i\ne0}
f\Bigl(
(s+\begin{psmallmatrix} 
j ( c j-2 b ) N
&
\frac{a \cc - 2 b \cc j N + c \cc j^2 N - c j p}{p} 
 \\
\frac{a \cc - 2 b \cc j N + c \cc j^2 N - c j p}{p}  &
\frac{a \cc^2 - 2 b \cc^2 j N + c \cc^2 j^2 N + 2 b \cc p - 2 c \cc j p}{p^2}
\end{psmallmatrix})
\tau+t_{0,i,0}\Bigr) \\
 %\text{\qquad (Type II terms)}
&\quad + p^{3k-6}\sum_{j\bmod{p}}
f(p(\ssp+w_j)\tau)
 %\text{\qquad (Type III terms)}
\\ &\quad 
+ p^{2k-6}\sum_{\substack{i,j \bmod{p}, \\ k\not\equiv 0 \bmod{p}}}
f\Bigl(
(s+\begin{psmallmatrix} 
j (-2 b + c j) N  &  
\frac{-a i + 2 b i j N - c i j^2 N - c j p}{p}   \\
\frac{-a i + 2 b i j N - c i j^2 N - c j p}{p} &  
\frac{a i^2 - 2 b i^2 j N + c i^2 j^2 N - 2 b i p + 2 c i j p}{p^2} 
\end{psmallmatrix})
\tau+t_{0,0,k}\Bigr)
\end{aligned}
\end{equation*}
\end{proposition}

\begin{proof}
Apply~\eqref{eqn:phisslash} to Proposition~\ref{prop:badprime-cosets}.
\end{proof}

We have the following available speed-ups.

\begin{proposition}\label{speedup-badprime}
Let $N,f,p,s,\ssp,\sSp$ be as at the beginning of this section.
Let $p\mmid N$.
For any integers $i,j,k$ let
$t_{i,j,k}=\begin{psmallmatrix}i/p&j/p\\j/p&k/p^2\end{psmallmatrix}$,
$u_{i,j,k}=\begin{psmallmatrix}i/p&j/p\\j/p&k/p^3\end{psmallmatrix}$,
$v_i=pt_{0,ia,i(ia+2b)}$,
and $w_j=j\begin{psmallmatrix}(-2b + jc)N/p&-c\\-c&0\end{psmallmatrix}$.
Then the following statements hold for all $e \in \Z_{\geq 0}$.
\begin{enumalph}
\item If $p \nmid a$, then
\begin{align*}
\Coeff_e& \sum_{i,j,k \bmod{p}}f(s\tau/p+t_{i,j,k})
=p\Coeff_e \sum_{j,k \bmod{p}}f(s\tau/p+t_{0,j,k}) \\
\intertext{and}
\Coeff_e&\sum_{i,j \bmod{p}}f((\ssp+w_j)\tau+t_{i,0,0})
=p\Coeff_e \sum_{j \bmod{p}}f((\ssp+w_j)\tau)
\end{align*}
\item If $p \nmid b$ then
$$
\Coeff_e \sum_{i,j,k \bmod{p}}f(s\tau/p+t_{i,j,k})
=p\Coeff_e \sum_{i,k \bmod{p}}f(s\tau/p+t_{i,0,k}).
$$
\item If $p \nmid c$ then
\begin{align*}
\Coeff_e&\sum_{i,j,k \bmod{p}}f(s\tau/p+t_{i,j,k})
=p\Coeff_e \sum_{i,j \bmod{p}}f(s\tau/p+t_{i,j,0}) \\
\intertext{and}
\Coeff_e&\sum_{i,k \bmod{p}}f((\sSp+v_i)\tau+t_{0,0,k})
=p\Coeff_e\sum_{i \bmod{p}}f((\sSp+v_i)\tau)
\intertext{and}
\Coeff_e&
\sum_{\substack{i,j \bmod{p}, \\ k \bmod{p^2}}}
f((\sSp+v_i)\tau/p+u_{0,j,k})
=p^2\Coeff_e\sum_{\substack{ i,j \bmod{p}}}f((\sSp+v_i)\tau/p+u_{0,j,0}).
\end{align*}
\item
For fixed $i$,
if $p\nmid ia+b$ then
\begin{align*}
\Coeff_e&\sum_{\substack{j \bmod{p}, \\ k \bmod{p^2}}}
f((\sSp+v_i)\tau/p+u_{0,j,k})
=p\Coeff_e\sum_{\substack{ k \bmod{p^2}}}
f((\sSp+v_i)\tau/p+u_{0,0,k}).
\end{align*}
\end{enumalph}
\end{proposition}

\begin{proof}
The proofs are similar to those of Proposition~\ref{speedup1}
\end{proof}

Another speed-up is that for $X,Y\in \M_2^{\text{\rm sym}}$,
if the set $\{t\in\Xtwo:\,\Tr(Xt)=e\}$ is empty
then $\Coeff_e f(X\tau+Y) = 0$.
Here the set is independent of~$Y$ and the conclusion holds for all~$Y$,
and so checking whether the set is empty can save significant
computation time.
Further, this result can be crucial when the denominator of a
particular formula for~$f$ might restrict to zero for some $X$ and
$Y$, because we simply skip this~$X$.

\bibliographystyle{plain}
\bibliography{bibola}

\end{document}